\documentclass[11pt,reqno]{amsproc}

\usepackage[sc]{mathpazo}\renewcommand{\mathbf}{\mathbold}
\normalfont
\usepackage[T1]{fontenc}

\usepackage[margin=1in]{geometry}
\usepackage{comment}
\usepackage[small,bf]{caption}
\usepackage{scalerel,nicefrac}
\usepackage[all,cmtip]{xy}
\usepackage{tikz-cd}
\usepackage{todonotes}

\usepackage{rotating, setspace}

\usepackage{bbm}
\usepackage{listings}
\usepackage{geometry}
\usepackage{amsmath,amsthm,amssymb}
\usepackage{hyperref}
\usepackage{color}
\usepackage{cleveref}
\usepackage{mathtools}
\usepackage{tikz-cd}
\usepackage{shuffle}
\usepackage{mathrsfs}
\usepackage{multicol}
\usepackage[utf8]{inputenc}
\usepackage{booktabs}
\usepackage{tabularx}
\usepackage{ragged2e}  %
\usepackage{longtable}
\usepackage{relsize}
\usepackage{thmtools}
\usepackage{thm-restate}
\usepackage{faktor}
\usepackage{stmaryrd}
\usepackage{enumerate}

\makeatletter
\renewcommand{\@secnumfont}{\bfseries}
\makeatother

\newcolumntype{Y}{>{\RaggedRight\arraybackslash}X} 

\hypersetup{
    colorlinks=true,
    linkcolor=blue,
    filecolor=blue,      
    }

\makeatletter
\def\l@subsection{\@tocline{2}{0pt}{2.5pc}{5pc}{}}

\makeatother
\numberwithin{equation}{section}

\newcommand{\N}{\mathbb{N}}

\newcommand{\R}{\mathbb{R}}
\newcommand{\F}{\mathbb{F}}
\newcommand{\E}{\mathbb{E}}
\newcommand{\C}{\mathbb{C}}

\newcommand{\cV}{\mathcal{V}}

\newcommand{\cH}{\mathcal{H}}
\newcommand{\cA}{\mathcal{A}}

\newcommand{\cW}{\mathcal{W}}
\newcommand{\cX}{\mathcal{X}}
\newcommand{\cM}{\mathcal{M}}
\newcommand{\cP}{\mathcal{P}}
\newcommand{\cF}{\mathcal{F}}
\newcommand{\cK}{\mathcal{K}}

\newcommand{\cE}{\mathcal{E}}

\newcommand{\fg}{\mathfrak{g}}

\newcommand{\fh}{\mathfrak{h}}

\newcommand{\fgl}{\mathfrak{gl}}

\newcommand{\fu}{\mathfrak{u}}

\newcommand{\bx}{\mathbf{x}}
\newcommand{\by}{\mathbf{y}}
\newcommand{\bz}{\mathbf{z}}

\newcommand{\bomega}{{\boldsymbol{\omega}}}

\newcommand{\bX}{\mathbf{X}}
\newcommand{\bY}{\mathbf{Y}}
\newcommand{\bw}{\mathbf{w}}
\newcommand{\bZ}{\mathbf{Z}}

\newcommand{\bu}{\mathbf{u}}
\newcommand{\bv}{\mathbf{v}}
\newcommand{\bq}{\mathbf{q}}

\newcommand{\obx}{\overline{\bx}}
\newcommand{\oby}{\overline{\by}}

\newcommand{\wbX}{\widetilde{\bX}}
\newcommand{\wbY}{\widetilde{\bY}}

\newcommand{\wbx}{\widetilde{\bx}}

\newcommand{\wbq}{\widetilde{\bq}}

\newcommand{\obX}{\overline{\bX}}
\newcommand{\obY}{\overline{\bY}}

\newcommand{\obu}{\overline{\bu}}
\newcommand{\obv}{\overline{\bv}}
\newcommand{\obw}{\overline{\bw}}
\newcommand{\obz}{\overline{\bz}}

\newcommand{\wQ}{\widetilde{Q}}

\newcommand{\tbX}{\widetilde{\bX}}

\newcommand{\GL}{\operatorname{GL}}

\newcommand{\sD}{\mathsf{D}}

\newcommand{\bF}{\mathbf{F}}

\newcommand{\cmGL}{\ensuremath{\textbf{\upshape GL}}}
\newcommand{\cmg}{{\boldsymbol{\mathfrak{g}}}}
\newcommand{\cmgl}{{\boldsymbol{\mathfrak{gl}}}}
\newcommand{\cmG}{\textbf{G}}

\newcommand{\gt}{\vartriangleright}

\newcommand{\ot}{\otimes}
\newcommand{\wot}{\,\widetilde{\ot}\,}
\newcommand{\op}{\operatorname{op}}
\newcommand{\id}{\text{id}}

\newcommand{\rank}{\text{rank}}
\newcommand{\Lip}{\operatorname{Lip}}

\newcommand{\Mat}{\operatorname{Mat}}
\newcommand{\thinhom}{\operatorname{th}}

\newcommand{\andd}{\text{and}}
\newcommand{\im}{\text{i}}
\newcommand{\td}{\widetilde{d}}

\newcommand{\ls}{\lesssim}

\newcommand{\trans}{\operatorname{trans}}

\newcommand{\Aut}{\text{Aut}}
\newcommand{\Der}{\text{Der}}

\newcommand{\V}{\R^d}
\newcommand{\pV}{\R^{d+2}}
\newcommand{\Lin}{\operatorname{L}}

\newcommand{\sE}{\mathsf{E}}
\newcommand{\bsE}{\ensuremath{ \textbf{\upshape\textsf{E}}}}

\newcommand{\tleq}{\operatorname{tl}} %

\newcommand{\homotopy}{{\boldsymbol{\eta}}}
\newcommand{\bighomotopy}{{\boldsymbol{\Xi}}}

\newcommand{\Fr}{\text{F}}

\newcommand{\cmb}{\delta}

\newcommand{\con}{\bomega}
\newcommand{\cona}{\alpha}
\newcommand{\conb}{\beta}
\newcommand{\conc}{\gamma}
\newcommand{\res}{\text{res}}

\newcommand*{\pmat}[1]{\begin{pmatrix}#1\end{pmatrix}}
\newcommand*{\vmat}[1]{\begin{vmatrix}#1\end{vmatrix}}
\newcommand*{\psmat}[1]{\begin{psmallmatrix}#1\end{psmallmatrix}}

\newcommand*{\pd}[2]{\frac{\partial #1}{\partial #2}}

\newcommand{\zor}{0} %

\newcommand{\bigcorner}{\scalebox{1.25}[1.25]{$\llcorner$}}
\newcommand{\raisecorner}[2]{\raisebox{0.25pt}{$#1\bigcorner$}}
\newcommand{\zax}{\mathpalette\raisecorner\relax} %

\newcommand{\zbdy}{\square} %

\newcommand{\dgcm}{\ensuremath{ \textbf{\upshape\textsf{D}}}}
\newcommand{\bdy}{\partial}

\newcommand{\concat}{\star}
\newcommand{\comp}{\star}
\newcommand{\dgcomp}{\odot}

\newcommand{\thinpath}{\mathsf{T}_1}
\newcommand{\thinsurface}{\mathsf{T}_2}
\newcommand{\wthinsurface}{\widetilde{\mathsf{T}}_2}
\newcommand{\thindg}{\ensuremath{\textbf{\upshape\textsf{T}}}}

\newcommand{\llb}{\llbracket}
\newcommand{\rrb}{\rrbracket}

\newtheorem{counter}{Counter}[section]
\newtheorem{theorem}[counter]{Theorem}
\newtheorem{lemma}[counter]{Lemma}

\newtheorem{proposition}[counter]{Proposition}
\newtheorem{corollary}[counter]{Corollary}

\theoremstyle{definition}
\newtheorem{definition}[counter]{Definition}
\newtheorem{example}[counter]{Example}
\newtheorem{remark}[counter]{Remark}

\crefname{equation}{}{}
\title{Random Surfaces and Higher Algebra}

\usepackage[foot]{amsaddr}

\makeatletter
\renewcommand{\email}[2][]{%
  \ifx\emails\@empty\relax\else{\g@addto@macro\emails{,\space}}\fi%
  \@ifnotempty{#1}{\g@addto@macro\emails{\textrm{(#1)}\space}}%
  \g@addto@macro\emails{#2}%
}
\makeatother

\author{Darrick Lee$^{1}$}
\email{darrick.lee@ed.ac.uk}
\address[$^1$]{School of Mathematics and Maxwell Institute, University of Edinburgh, Edinburgh EH9 3FD, Scotland}

\author{Harald Oberhauser$^{2}$}
\address[$^2$]{Mathematical Institute, University of Oxford, Andrew Wiles Building, Radcliffe Observatory Quarter,
Woodstock Rd, Oxford OX2 6GG}
\email{oberhauser@maths.ox.ac.uk}

\begin{document}

\begin{abstract}

We introduce a characteristic function for laws of random surfaces $\bX: [0,s] \times [0,t] \to \R^d$, in the spirit of expected path developments for one-dimensional stochastic processes into matrix groups.
A key property is that path development is structure preserving: path concatenation becomes matrix multiplication. 
The main challenge is to account for two distinct concatenation operations for surfaces: horizontal and vertical. 
To address this, we use the notion of surface holonomy from higher geometry to define surface developments, and study this in a stochastic context. 
We generalize surface developments to the Young setting of $\rho$-H\"older surfaces, where $\rho > \frac12$, show that such developments characterize parametrized surfaces.
Our main result shows that the resulting expected surface development provides a computable and structured description of laws of random surfaces and leads to a natural metric on the space of probability measures on surfaces.

\end{abstract}
\maketitle

{\footnotesize \tableofcontents} %

\section{Introduction}
\addtocontents{toc}{\protect\setcounter{tocdepth}{1}}

The main contribution of this article is to develop a structured description of probability measures on the function space
\begin{equation}\label{eq:fct space}
\bigcup_{s>0, t>0}C([0,s]\times [0,t],\mathbb{R}^d).
\end{equation}
In fact, we first study this question on the space of probability measures on the quotient space $\mathcal{X} \coloneqq C([0,1]^2,\mathbb{R}^d) /\sim$; here, $\sim$ denotes an equivalence relation that roughly identifies surfaces up to reparametrization and folding~\cite{bischoff_thin_2025}.
The above case of probability measures on the space of functions \eqref{eq:fct space} follows as a corollary.
Moreover, understanding the case of equivalence classes is firstly, natural from our geometric approach and secondly also potentially important for applications; e.g.~an image stretched or shrunk contains the same information. 
Further, this is analogous to how one proceeds with the expected path development.

\subsection*{A Geometric Approach.} 
Our approach to construct an expected surface development proceeds in two steps.
The first step is the construction of a parametrized family of maps 
\[
\bF^{\alpha, \gamma} = (F^{\alpha}, H^{\alpha, \gamma}) :\cX \to \dgcm
\]
into an object $\dgcm$ called a double group.
Here, we primarily work with $H^{\alpha, \gamma}$ which encodes information about the entire surface, while $F^\alpha$ encodes information about boundary of the surface, which is necessary to define the compositional structure.
Double groups generalize groups to encode two independent compositions.
Surfaces compose horizontally and vertically, and any representation must preserve both operations.
Because we care about computability, we use a special double group that can be realized by matrices.
To do all this, we revisit results from higher geometry and higher algebra. 
The second step is to combine the map $\bF^{\alpha, \gamma}$ with probabilistic estimates. 
The main result will be that the map 
 \begin{align}
        \mu \mapsto \left(\E_{\bX \sim \mu} \left[\cH^{\alpha,\gamma,\ell}(\bX) \right]\right)_{\alpha,\gamma,\ell} \quad \text{where} \quad \cH^{\alpha, \gamma,\ell}(\bX)\coloneqq \exp(i \langle \ell, H^{\alpha, \gamma}(\bX) \rangle)
\end{align}
characterizes a rich class of probability measures on $\cX$. Here we take $\cH^{\alpha, \gamma,\ell}(\bX)$ since tensor products of higher representations are not well-behaved (see~\Cref{apxsec:2_tensor_product}); therefore we must pass to scalar-valued observables via exponentiation.
Furthermore, this leads to a natural metric on the space of probability measures on surfaces.
Below we give a few details and emphasize the analogy with the well-studied case of the expected path development.

\subsection*{Characteristic Functions.}
Our aim is to construct a characteristic function (or Fourier transform) for surface-valued random variables.
Let us recall, some well-known examples of characteristic functions, going from linear spaces to various non-linear spaces.

\begin{description}
     \item[Euclidean]
 Suppose $\mu$ is a probability measure on $\cX =\R^d$. In this case, the Fourier transform is given by
\begin{align}
    \mu \mapsto (\E_{x \sim \mu} [F^\alpha(x)])_{\alpha \in \R^d}, \quad F^\alpha(x) \coloneqq \exp(i \langle \alpha, x \rangle).
\end{align}
Here, the $F^\alpha$ is an irreducible unitary representation of $\R^d$ into $U(1)$, where $\alpha \in \R^d$ parametrizes all such representations. \medskip

\item[Group] 
In the generalized setting where $\cX = G$ is a compact Hausdorff topological group, classical results from noncommutative harmonic analysis tells us that the analogous construction holds. In particular, the Fourier transform of a $G$-valued probability measure $\mu$ is given by 
\begin{align}
    \mu \mapsto (\E_{x \sim \mu} [F^\alpha(x)])_{\alpha \in M},
\end{align}
where $M$ denotes the set of isomorphisms classes of irreducible unitary representations $F^\alpha: G \to U^\alpha$, where $U^\alpha$ is the unitary group of a finite dimensional Hilbert space $V^\alpha$.\medskip

\item[Paths] 
Now we let $\cX$ be a subset of $C([0,1],\R^d)$ that is sufficiently regular such that the following differential equation 
\begin{align}\label{eq:CDEsimple}
    \frac{dF^\alpha_t(\bx)}{dt} = F_t^\alpha(\bx) \cdot \alpha \left( \frac{d\bx_t}{dt}\right), \quad F^\alpha_0(\bx) = I,
\end{align}
is well-posed when $\alpha \in L(\R^d, \fu(n))$. This is parallel transport on a trivial principal $U(n)$ bundle with a translation-invariant connection $\alpha$.
We then consider
\begin{equation}\label{eq:exp sig}
\mu \mapsto \left(\mathbb{E}[F_1^\alpha(\bX)] \right)_{\alpha \in M}
\end{equation}
as a characteristic function for probability measures on pathspace
where 
\begin{align}
    M \coloneqq \{ \alpha \in L(\R^d, \fu^n) \, : \, n \in \N \}.
\end{align}
This construction was rediscovered by different communities at different times in different contexts and aspects of it run under various names such "non-commutative Fourier transform"~\cite{kapranov_noncommutative_2009} or "expected path signature"~\cite{baudoin_introduction_2004}. \medskip
\end{description}
The unifying theme of all the above, is that the domain is a group and if we have a representation of the group given by $F^\alpha$ then we expect the map
\begin{align} \label{eq:general_characteristic_function}
    \mu \mapsto (\E_{x \sim \mu} [F^\alpha(x)])_{\alpha \in M},
\end{align}
to be injective when $M$ parametrizes a large enough set of such representations. 
The goal of this article is to carry out a similar construction for probability measures on surfaces. There are two crucial properties which underlies the characteristic function.
\begin{enumerate}[(i)]
    \item \textbf{Algebraic Structure.} The key point which underlies many of the desirable properties of a characteristic function is the fact that $F^\alpha$ is a representation: it respects the algebraic structure of the underlying space.
    When $\cX = \R^d, G$, these are simply group homomorphisms. 
    For paths, we consider the algebraic structure of concatenation, and note that any path gives rise to another path by running it backwards, that is a time-reversal. 
    This becomes a group after considering thin homotopy equivalence $\sim_{\thinhom}$ (in stochastic analysis, this is called tree-like equivalence); informally, thin-homotopy ignores time-changes and retracings.
    One can verify that \eqref{eq:CDEsimple}, $F^\alpha$ respects these operations: path concatenation and time reversal turn into group multiplication and inverses in the codomain under $F^\alpha$.
    
    \item \textbf{Point Separation.} To make the map \eqref{eq:general_characteristic_function} injective, the collection of representations $(F^\alpha)_{\alpha \in M}$ must separate points: for $x\neq y$ there exists an $\alpha \in M$ such that $F^\alpha(x)\neq F^\alpha(y)$. Otherwise we would not even be able to distinguish Dirac measures on $x$ and $y$.
    This point separation property holds in the case of Euclidean space and groups, but it is not directly true for the path case. 
    In particular, the construction \eqref{eq:CDEsimple} is invariant under reparametrizations; thus one cannot 
    distinguish paths that differ by a time change, $\bx$ and $\by_t =\bx_{\phi(t)}$.
    Thus, $F^\alpha$ can separate thin homotopy classes of paths~\cite{chevyrev_characteristic_2016}. By appending the time-coordinate to paths\footnote{That is, if $x \neq y$ then $t \mapsto (t,x(x))$ and $t \mapsto (t, y(t))$ are not thin-homotopy equivalent.}, we can separate points on the entire path space. 
\end{enumerate}

\subsection*{Contribution.}
The notion of surface holonomy from higher geometry~\cite{baez_higher_2004} provides us with a candidate $F^\alpha$ for the representation of surfaces, which respects their higher-dimensional compositional structure. %
As our focus is on computable representations, we restrict our attention to a class of surface holonomy valued in a higher analogue of the classical matrix groups, such that they can be determined by standard matrix operations. 
Our main contributions establish analytic and probabilistic properties of surface holonomy to develop a concrete, computable tool for studying random surfaces in stochastic analysis and applications. In particular, we consider the following properties, and emphasize that these were only established for paths in recent years.
\begin{enumerate}
    \item \textbf{Characteristic Functions.} We show that surface holonomy can be used to build a characteristic function for random surfaces. However, due to the algebraic structure of higher tensor products, we must modify the heuristic in~\eqref{eq:general_characteristic_function} and consider \emph{exponentials} of representations. We emphasize that unlike the case of paths, there are already many examples of smooth random surfaces which are widely studied~\cite{adler_random_2009}.
    \item \textbf{Metrics.} In the case of paths, these characteristic functions have been used to define computable metrics for laws of stochastic processes~\cite{lou_pcf-gan_2023}. We show that we can define analogous metrics for the laws of random surfaces.
    \item \textbf{Separation of Surfaces.} Recent work has shown that surface development can also separate surfaces up to thin homotopy in the piecewise linear setting~\cite{bischoff_thin_2025}. Here, we show that appending the parametrization allows us to separate surfaces.
    \item \textbf{Nonsmooth Surfaces and Continuity.} Prior work on surface holonomy has focused exclusively in the smooth setting. Here, we consider a generalization to $\rho$-H\"older surfaces in the Young regime $(\rho > \frac12)$ in order to handle irregular random surfaces such as fractional Brownian sheets (with Hurst parameter $h > \frac12$). We show that the above properties also hold in this Young regime. 
\end{enumerate}

Our main contributions can be summarized in the following informal theorem.

\begin{theorem} \label{thm:intro_main_contribution}
    Let $\rho > \frac12$ and let $\cP^\rho$ denote the space of probability measures valued in the $\rho$-H\"older surfaces $C^\rho([0,1]^2, \V)$. The \emph{surface development characteristic function (SDCF)}
    \begin{align}
        \left( \mu \mapsto \E_{\bX \sim \mu} \left[\cH^{\alpha, \gamma, \ell}(\obX) \right]\right)_{\alpha, \gamma, \ell}
    \end{align}
    separates probability measures in $\cP^\rho$. Furthermore, there exists a metric $d: \cP^\rho \times \cP^\rho \to \R$ defined by the SDCF which metrizes the weak topology on compact subsets $\cK \subset \cP^\rho$.
\end{theorem}
\subsection*{Outline.}
Section \ref{sec: path dev} presents the well-known case of stochastic processes that are indexed by a one-dimensional set. That is, we discuss the above mentioned path-development \eqref{eq:CDEsimple} and characteristic function \eqref{eq:exp sig}. 
Here, the domain and co-domain are a group so the  representation $\Phi^\alpha$ is a group morphism.
Section \ref{sec:surfaces} introduces the algebraic structure of surfaces as that of a double group. Motivated by the path-case, our aim is to construct a double group morphism to represent surfaces.
Section \ref{sec:matrix_crossed modules} constructs the co-domain of the sought double group morphism from the previous section as a crossed modules of matrices; ultimately this allows for computations that are easier to manipulate symbolically.
Section \ref{sec:surface_development} explicitly constructs a double group morphism from surfaces into a crossed module of matrices as a surface development along a 2-connection; this is the generalization of the path development \eqref{eq:CDEsimple} to surfaces and provides us with a representation to construct a characteristic function for random surfaces.
Section \ref{sec:Young} generalizes surface development of the previous section to non-smooth surfaces in the Young regime.
Section \ref{sec:parametrized} shows how to go from a representation of (thin-homotopy) equivalences classes of surfaces to a representation of parametrized surfaces.
Section \ref{sec:characteristicness} then combines the surface development of the previous sections with probabilistic estimates to define the characteristic function for surfaces and a natural metric for laws of random surfaces. \medskip

\subsection*{Related Work.}
Functorial representations of paths via~\Cref{eq:CDEsimple} have been deeply studied in many communities, often under a distinct focus and name. For instance, this has been studied in analysis as \emph{product integrals}~\cite{slavik_product_2007,masani_multiplicative_1947,schlesinger_parallelverschiebung_1928}, in physics as \emph{path-ordered exponentials}, in control theory as \emph{Chen-Fliess expansions}~\cite{fliess_fonctionnelles_1981}, in geometry as \emph{parallel transport} or \emph{holonomy}, and in rough paths as the \emph{path signature} or \emph{path development}. The main perspectives used in this article are the latter two.

The theory of rough paths~\cite{lyons_differential_1998,lyons_differential_2007,friz_multidimensional_2010,friz_course_2020} initiated the use of signatures in stochastic analysis to understand stochastic processes.
More recently, there has been interest in extending signatures to higher dimensional data.
In~\cite{chouk_rough_2014,chouk_skorohod_2015}, controlled rough paths are extended to the case of surfaces and~\cite{alberti_integration_2023,stepanov_towards_2021} develops an integration theory of rough differential forms. 
In~\cite{giusti_topological_2022}, the authors develop a topological approach to higher dimensional signatures in a continuous setting, and in the discrete setting~\cite{diehl_two-parameter_2022} follows an algebraic approach. 
In contrast to the generalization we present, these generalizations of the path signature do not respect the compositional structure of surfaces, that is, they are not functorial.

To develop a functorial representation of surfaces, we turn to the geometric perspective. 
The formalization of path holonomy in terms of a group homomorphism or functor is developed in~\cite{barrett_holonomy_1991-1,caetano_axiomatic_1994,schreiber_parallel_2009}. Perhaps the earliest study of a surface analogue of parallel transport leads back to the work of Schlesinger~\cite{schlesinger_parallelverschiebung_1928}
More recently, the mathematical physics community constructed higher parallel transport, which generalized the previous surface integrals of curvature to nontrivial 2-groups, to develop higher gauge theory~\cite{baez_higher_2004,baez_higher_2006,bartels_higher_2006,girelli_higher_2004,breen_differential_2005,martins_surface_2011,schreiber_smooth_2011}, and this is the approach we take. 

After the first version of this paper was posted, several related papers have also appeared. Based on~\cite{kapranov_membranes_2015}, a universal version of the surface signature was developed in~\cite{lee_surface_2024, chevyrev_multiplicative_2024-1}, which also holds for rough surfaces. Furthermore~\cite{bischoff_thin_2025} show that the surface signature characterizes surfaces up to thin homotopy for piecewise linear surfaces.

\medskip

\textbf{Terminology.} We use the terms \emph{path development} and \emph{surface development} to refer to path/surface holonomy with respect to translation-invariant connections, in order to reflect the common terminology in stochastic analysis. \medskip

{\small
\textbf{Acknowledgements.}
The authors would like to thank the organizers, Kurusch Ebrahimi-Fard and Fabian Harang, of the \emph{Structural Aspects of Signatures and Rough Paths} workshop, where this work was first presented and many helpful discussions took place. We would like to thank Chad Giusti, Vidit Nanda, Francis Bischoff, Camilo Arias Abad, Jo\~ao Faria Martins for several helpful discussions. We would also like to thank Luca Bonengel for proofreading the first version of this article.
HO and DL were supported by the Hong Kong Innovation and Technology Commission (InnoHK
Project CIMDA). HO was supported by the EPSRC, [grant number EP/S026347/1] and the "Erlangen Programme" for AI [grant number EP/Y028872/1].}

\section{Path Development}\label{sec: path dev}
We recall the classic path development in a way that makes the surface development clearer. We let $C^\infty([0,1], \R^d)$ denote the space of smooth paths with \emph{sitting instants}\footnote{Note that a smooth path can always be smoothly reparametrized to have sitting instants. However, this is required so that the concatenation of two paths is still smooth.}: there exist some $\epsilon > 0 $ such that $\bx_0 = \bx_s$ and $\bx_1 = \bx_{1-s}$ for all $s \in [0,\epsilon] \label{pg:smooth_paths}$.
We consider \emph{based} paths \emph{up to translation},
\begin{align}
    C^\infty_0([0,1],\V) = C^\infty([0,1],\V)/\sim_{\trans}
\end{align}
where two maps $f,g : D \to \V \label{pg:based}$ are \emph{translation equivalent}, $f \sim_{\trans} g$, if $f = g + v$ for some fixed $v \in \V$. 
For brevity, we refer to elements of $C^\infty_0([0,1],\V)$ as paths.
\subsection*{Operations with Paths}
Paths carry two natural operations: concatenation and time-reversal. 
Formally, we define the concatenation operator
\begin{align} \label{eq:path_concatenation}
    (\bx \concat \by)_t \coloneqq \left\{ \
    \begin{array}{cl}
        \bx_{2t} & : t \in [0, \frac12] \\
        \bx_1 + \by_{2t-1} &: t \in (\frac12, 1]
    \end{array}
    \right.
    \quad \text{for} \quad  \bx, \by \in C^\infty_0([0,1], \R^d)
\end{align}
and time reversal as the inverse,
\begin{align} \label{eq:path_inverse}
    \bx_t^{-1} \coloneqq \bx_1 - \bx_{1-t} \quad \text{for} \quad \bx \in C^\infty_0([0,1], \R^d). 
\end{align}
\subsection*{Paths as a Group}
It is tempting to think of concatenation as a (non-commutative) group multiplication and time-reversal as an inverse. 
However, concatenation is not associative, and time reversal does not lead to an inverse on $C^\infty_0([0,1], \R^d)$ itself. 
But if we identify paths under reparametrizations and retracings, then we get the desired group structure.
\begin{figure}[!h]
    \includegraphics[width=\linewidth]{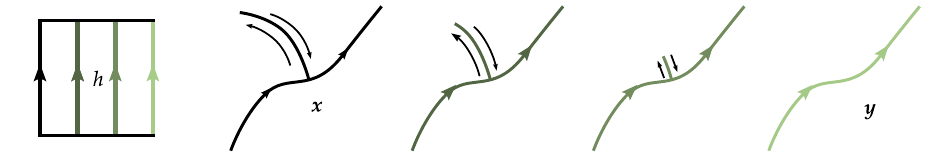}
\end{figure}

\begin{definition} \label{def:path_thinhom}
    Let $\bx, \by \in C^\infty_0([0,1], \R^d)$ such that $\bx_1 = \by_1$. We say that $\bx$ and $\by$ are \emph{thin homotopy equivalent}, denoted $\bx \sim_{\thinhom} \by$ if there exists a smooth endpoint-preserving thin homotopy $\homotopy: [0,1]^2 \to \R^d$ between $\bx$ and $\by$, where
    \begin{itemize}
        \item \textbf{(endpoint preserving)} $\homotopy_{s,0} = 0$ and $\homotopy_{s,1} = \bx_1$;
        \item \textbf{(homotopy condition)} $\homotopy_{0,t} = \bx_t$ and $\homotopy_{1,t} = \by_t$; and
        \item \textbf{(thinness condition)} $\rank(d\homotopy) \leq 1$, where $d\homotopy$ is the differential of $\homotopy$.
    \end{itemize}
\end{definition}

Note that thin homotopy equivalence is the same as tree-like equivalence for smooth paths~\cite{tlas_holonomic_2016,meneses_thin_2021}. 
The space of thin homotopy (and translation) equivalence classes of paths
\begin{align} \label{eq:thin_paths}
    \thinpath(\V) \coloneqq C^\infty_0([0,1], \V)/\sim_{\thinhom}
\end{align}
is a group under concatenation~\Cref{eq:path_concatenation} and time reversal for the inverse~\Cref{eq:path_inverse}.

\subsection*{Path Development}
Next, we define path development as parallel transport with respect to a translation-invariant connection.
\begin{definition}
    Let $G$ be a Lie group with Lie algebra $\fg$. A \emph{(translation-invariant) connection} is a linear map $\cona \in \Lin(\V, \fg)$. For a path $\bx \in C^\infty_0([0,1], \V)$, we define the \emph{path development of $\bx$ with respect to $\alpha$} as the solution at $t=1$ of the differential equation for $F^\alpha_t(\bx): [0,1] \to G \label{pg:path_development}$, defined by
    \begin{align}
        \frac{dF^\alpha_t(\bx)}{dt} = F^\alpha_t(\bx) \cdot \alpha \left( \frac{d\bx_t}{dt}\right), \quad F^\alpha_0(\bx) = e,
    \end{align}
    where $e \in G$ is the identity. We define $F^\alpha(\bx) = F^\alpha_1(\bx)$. 
\end{definition}
In fact, path development is invariant under thin homotopy~\cite[Theorem 2]{caetano_axiomatic_1994}, and preserves composition; thus for any connection $\alpha$, we obtain a group homomorphism
\begin{align}
    F^\alpha : \thinpath(\V) \to G.
\end{align}
For analytic applications, we often take $G \subset \GL^n$ to be a matrix Lie group. In fact, such path developments into unitary groups is point-separating.
\begin{theorem}{\cite[Theorem 4.8]{chevyrev_characteristic_2016}}
    Let $\bx, \by \in \thinpath(\V)$ such that $\bx \neq \by$. Then, there exists some $n \geq 1$ and a connection $\alpha \in \Lin(\V, \fgl^n)$ such that $F^\alpha(\bx) \neq F^\alpha(\by)$.
\end{theorem}
This result immediately implies an analogous result for \emph{parametrized paths}, where $\obx_t = (t, \bx_t)$.

\begin{corollary} \label{cor:path_separation}
    Let $\bx, \by \in C^\infty_0([0,1], \V)$ such that $\bx \neq \by$. Then, there exists some $n \geq 1$ and a connection $\alpha \in \Lin(\R^{d+1}, \fgl^n)$ such that $F^\alpha(\obx) \neq F^\alpha(\oby)$.
\end{corollary}

\subsection*{Probabilistic Statements}
Following the introduction, we define a characteristic function for measures $\mu$ of unparametrized (thin homotopy classes of) paths, and $\nu$ of parametrized paths by
\begin{align} \label{eq:probabilistic_statements_char_function}
     \mu \mapsto \left(\E_{\bx \sim \mu} [F^\alpha(\bx)]\right)_{n \geq 0, \, \alpha \in \Lin(\R^d, \fgl^n)} \quad \andd \quad \nu \mapsto \left(\E_{\bx \sim \nu}[F^\alpha(\obx)]\right)_{n \geq 0, \,\alpha \in \Lin(\R^{d+1}, \fgl^n)}.
\end{align}
The fact that this separates measures was shown in~\cite{chevyrev_characteristic_2016, cuchiero_global_2023} for measures satisfying certain moment decay conditions\footnote{While these references state their results for the path signature, we give the formulation in terms of path developments by using the universal property of the path signature, see~\cite{chevyrev_characteristic_2016}.}. Another approach was studied in~\cite{chevyrev_signature_2022} which holds without decay conditions, but instead applying a normalization procedure. 
However,~\cite[Remark 5.5(iv)]{cuchiero_global_2023} notes that one can obtain characteristic functions on the entire path space \emph{without} either moment conditions or normalization by instead considering %
\begin{align} \label{eq:probabilistic_statements_reformulated}
     \E_{\bx \sim \mu} \Big[\exp(i \langle \ell, F^\alpha(\bx) \rangle)\Big]
\end{align}
where we view $\ell \in \Mat_{n,n}$ as linear functionals. 
Analytic expressions for such characteristic functions can be obtained for solutions to certain SDEs~\cite{cuchiero2023signature}. We aim to generalize this notion of a characteristic function for surfaces. \medskip

Furthermore, the characteristic functions of the form in~\Cref{eq:probabilistic_statements_char_function} have been used to define metrics for stochastic processes~\cite{lou_pcf-gan_2023}; used for generative time series modeling in machine learning. This metric is obtained by randomizing the selection of the connection. 
We can reformulate this in terms of~\Cref{eq:probabilistic_statements_reformulated} for the parametrized setting as follows.

\begin{theorem}{\cite{lou_pcf-gan_2023, cuchiero2023signature}}\label{thm:haometric}
    Let $\Xi^n$ be the Gaussian measure on $\Lin(\R^{d+1}, \fgl^n)$, $\Theta^n$ be the uniform measure on the unit ball in $\Mat_{n,n}$ and $\cP$ be Borel probability measures on $C^\infty_0([0,1], \V)$. Then $d: \cP^2 \to \R$ defined by
    \begin{align}
        d(\mu, \nu) = \sum_{n=1}^\infty \E_{\alpha \in \Xi^n} \E_{\ell \in \Theta^n} \left[ \left| \E_{\bx \sim \mu} \Big[\exp(i \langle \ell, F^\alpha(\obx) \rangle)\Big] - \E_{\by \sim \nu} \Big[\exp(i \langle \ell, F^\alpha(\oby) \rangle)\Big]\right|\right]
    \end{align}
    is a metric which metrizes the weak topology when restricted to a compact subset $\cK \subset \cP$.
\end{theorem}

One of our main contributions, Theorem \ref{thm:intro_main_contribution}, is to generalize the characteristic function \eqref{eq:probabilistic_statements_reformulated} and the metric in Theorem \ref{thm:haometric} from probability measures on paths to probability measures on surfaces.

\section{Algebraic Structure of Surfaces}\label{sec:surfaces}
In this section we recall the classic surface development; we follow the structure of the previous section on path development to emphasize the motivation and similarities. We denote $C^\infty([0,1]^2, \V) \label{pg:smooth_surfaces}$ to be smooth surfaces with \emph{sitting instants}: there exists some $\epsilon > 0$ such that 
\begin{align}
    \bX_{u,t} = \bX_{0,t}, \quad \bX_{1-u, t} = \bX_{1,t}, \quad \bX_{s,u} = \bX_{s,0}, \quad \bX_{s,1-u} = \bX_{s,1} \quad \text{for all} \quad u \in [0,\epsilon],\, s,t \in [0,1].
\end{align}
Furthermore, we let $C^\infty_0([0,1]^2,\V)$ denote translation equivalence classes of surfaces, represented by surfaces $\bX$ based at the origin $\bX_{0,0} = 0$.

\subsection*{Operations with Surfaces}
We equip $C^\infty_0([0,1]^2,\V)$ with \emph{partially defined} \emph{horizontal} and \emph{vertical} concatenation operators, 
\begin{align} \label{eq:surface_concatenation}
    (\bX \concat_h \bY)_{s,t} \coloneqq \left\{ \
    \begin{array}{cl}
        \bX_{2s,t} & : s \in [0, \frac12] \\
        \bX_{1,0} + \bY_{2s-1, t} &: s \in ( \frac12, 1]
    \end{array}\right.
    , \, 
    (\bX \concat_v \bY)_{s,t} \coloneqq \left\{ \
    \begin{array}{cl}
        \bX_{s, 2t} & : t \in [0,  \frac12] \\
        \bX_{0,1} + \bY_{s, 2t-1} &: t \in ( \frac12, 1].
    \end{array}\right.
\end{align}
These operations are only well defined if $\bX_{1,t} = \bX_{1,0} + \bY_{0,t}$ for all $t \in [0,1]$ ($\bX$ and $\bY$ are \emph{horizontally composable}), and if $\bX_{s,1} = \bX_{0,1} + \bY_{s,0}$ for all $s \in [0,1]$ ($\bX$ and $\bY$ are \emph{vertically composable}), respectively. 
Given $\bX\in C^\infty_0([0,1]^2, \V)$, the \emph{left, right, bottom and upper boundary maps} are
\begin{align} \label{eq:strict_boundaries}
    (\bdy_l \bX)_t \coloneqq \bX_{0,t}, \quad (\bdy_r \bX)_t \coloneqq \bX_{1,t}, \quad (\bdy_b \bX)_t \coloneqq \bX_{t,0}, \quad (\bdy_u \bX)_t \coloneqq \bX_{t,1}.
\end{align}
For any path $\bx \in C^\infty_0([0,1], V)$, we define the horizontal and vertical identities of $\bx$, denoted $1^h_{\bx}, 1^v_{\bx} \in C^\infty_0([0,1]^2,\V)$ respectively, by
\begin{align} \label{eq:strict_identites}
    (1^h_{\bx})_{s,t} \coloneqq \bx_t \quad \text{and} \quad (1^v_\bx)_{s,t} \coloneqq \bx_s.
\end{align}

For any smooth surface $\bX\in C^\infty_0([0,1]^2,\V)$, we define the horizontal and vertical inverses of $\bX$, denoted $\bX^{-h}, \bX^{-v} : [0,1]^2 \to \R^d$ by
\begin{align} \label{eq:strict_inverses}
    \bX^{-h}_{s,t} = \bX_{1-s,t} \quad \text{and} \quad \bX^{-v}_{s,t} = \bX_{s,1-t}.
\end{align}

\begin{remark}
    One can define a general concatenation of two surfaces $\bX: D_X \to \V$ and $\bY: D_Y \to \V$, where $D_X, D_Y \subset \R^2$ are two general domains, into a surface $\bX \concat \bY : D \to \V$, by specifying an inclusion $\iota: D_X \sqcup D_Y \rightarrow D$ such that the overlap occurs on the boundary of $D_X$ and $D_Y$. However, such surfaces are thin homotopy equivalent (\Cref{def:2d_thin_homotopy}) to the horizontal/vertical concatenations of surfaces on rectangular domains under mild conditions on the boundaries. Thus, there is no loss of generality by considering horizontal/vertical concatenations.
\end{remark}

\subsection*{Double Groups.}
We now formalize the notion of a double group: an algebraic object with two multiplication and inverse operations. This is a special case of a double groupoid, which has been widely studied in category theory.
\begin{definition}
    A \emph{double group} $\dgcm = (\sD_1, \sD_2)$ consists of a group $\sD_1$ of \emph{edges}, and a set $\sD_2$ of \emph{squares} equipped with left, right, bottom, and upper boundary maps $\bdy_l, \bdy_r, \bdy_b, \bdy_u : \sD_2 \to \sD_1$. In particular, an element $X \in \sD_2$ can be visualized as follows, where $w,x,y,z \in \sD_0$ denote the four corresponding boundaries. 
    \begin{figure}[!h]
        \includegraphics[width=\linewidth]{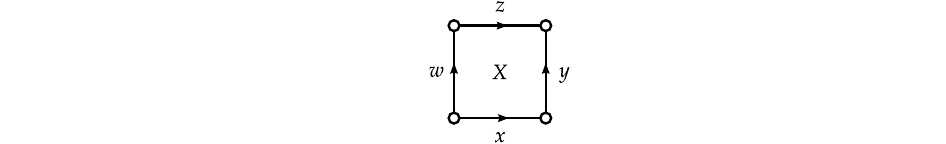}
    \end{figure}
    
    \noindent The squares in $\sD_1$ are equipped with two partial compositions. Let $X, X' \in \sD_2$.
    \begin{itemize}
        \item \textbf{(Horizontal Composition.)} If $\bdy_r X = \bdy_l X'$, there exists a composite $X \comp_h X'$.
        \item \textbf{(Vertical Composition.)} If $\bdy_u X = \bdy_b X'$, there exists a composite $X \comp_v X'$.
    \end{itemize}
    \begin{figure}[h]
        \includegraphics[width=\linewidth]{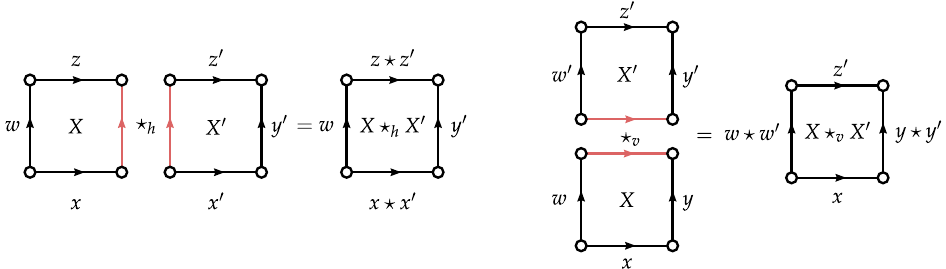}
    \end{figure}

    For any $X,Y,Z,W \in \sD_2$, which are appropriately composable, the \emph{interchange law} holds,
    \begin{align}
        (X \comp_h Y) \comp_v (W \comp_h Z) = (X \comp_v W) \comp_h (Y \comp_v Z),
    \end{align}
    Finally, there exist units and inverses for both compositions.
    \begin{itemize}
        \item \textbf{(Identity Squares.)} For any $x \in \sD_1$, there exists a \emph{horizontal identity} $1^h_x \in \sD_2$ and a \emph{vertical identity} $1^v_x \in \sD_2$ which acts as a unit under composition for horizontal and vertical composition respectively. In particular, for $X,Y \in \sD_2$ such that $\bdy_l X = \bdy_r Y = x$, we have $1^h_x \comp_h X = X$ and $Y \comp_h 1^h_x = Y$, and same for vertical composition.
        
        \item \textbf{(Isomorphism.)} For any $X \in \sD_2$, there exist horizontal and vertical inverses, denoted $X^{-h}, X^{-v} \in \sD_2$ respectively. In particular, we have $X \comp_h X^{-h} = 1^h_{\bdy_l X}$ and $X^{-h} \comp_h X = 1^h_{\bdy_r X}$, and similarly for $X^{-v}$. 
        \begin{figure}[h]
            \includegraphics[width=\linewidth]{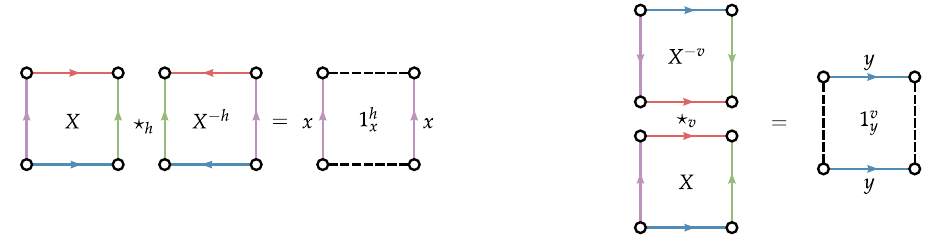}
        \end{figure}
    \end{itemize}
\end{definition}

\subsection*{Thin Homotopy and Unparametrized Surfaces as a Double Group}
As in the path case, one can view both horizontal and vertical composition as a partially defined group multiplication and the horizontal and vertical reversals as inverses.  
A minimal requirement is associativity for both horizontal and vertical composition.
However, these fail to hold due to reparametrization issues, as in the 1D setting.
Further, the horizontal composition of a surface $\bX$ with its horizontal inverse $\bX^{-h}$ results in a ``fold'' (and not the identity square).

In the 1D case, both these issues (non-associativity and reversal yielding not the identity) were resolved by identifying paths up to thin-homotopy. 
Hence, we arrive at the following direct generalization of thin homotopy.
\begin{figure}[!h]
    \includegraphics[width=\linewidth]{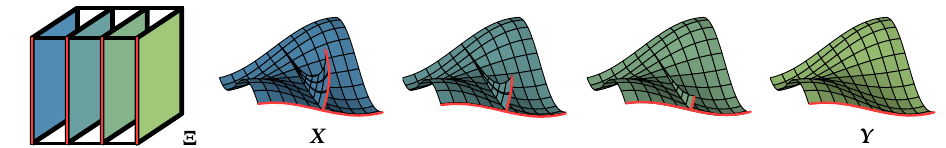}
\end{figure}    

\begin{definition}{\cite{martins_surface_2011}} \label{def:2d_thin_homotopy}
    Suppose $\bX, \bY \in C^\infty_0([0,1]^2,\V)$ are surfaces with equal corners, where $\bX_{i,j} = \bY_{i,j}$ for all $i,j \in \{0,1\}$. We say that $\bX$ and $\bY$ are \emph{thin homotopy equivalent}, denoted $\bX \sim_{\thinhom} \bY$ if there exists a smooth corner-preserving thin homotopy $\bighomotopy: [0,1]^3 \to \R^d$ between $\bX$ and $\bY$, where
    \begin{itemize}
        \item \textbf{(homotopy condition)} $\bighomotopy_{0,s,t} = \bX_{s,t}$ and $\bighomotopy_{1,s,t} = \bY_{s,t}$;
        \item \textbf{(thin homotopy boundaries)} the four sides of the homotopy $\bighomotopy_{u,s,0}, \bighomotopy_{u,s,1}, \bighomotopy_{u,0,t}, \bighomotopy_{u,1,t}$ are thin homotopies between the four boundary paths of $\bX$ and $\bY$; and
        \item \textbf{(thinness condition)} $\rank(d\bighomotopy) \leq 2$, where $d\bighomotopy$ is the differential of $\bighomotopy$.
    \end{itemize}
    We denote the set of thin homotopy (and translation) classes of surfaces by
    \begin{align} \label{eq:thin_surfaces}
        \thinsurface(\V) \coloneqq C^\infty_0([0,1]^2, \V)/\sim_{\thinhom}
    \end{align}
    and refer to the elements of $\thinsurface(V)$ as unparametrized surfaces.
\end{definition}
Thinness for surfaces intuitively means the homotopy sweeps no volume.
Thin homotopy classes of surfaces, indeed gives our first example of a double group. %

\begin{theorem}[\protect{\cite[Theorem 2.13]{faria_martins_fundamental_2011}}]
\label{def:thin_double_group} 
   The set $\thindg(\V) = (\thinpath(\V), \thinsurface(\V))$ of thin homotopy classes of paths~\Cref{eq:thin_paths} and of surfaces~\Cref{eq:thin_surfaces} is a double group.
   We call $\thindg(\V)$ the \emph{thin double group} on $\V$. 
    \end{theorem}%
Suppose $[\bX], [\bY] \in \thinsurface(\V)$ are two thin homotopy equivalence classes of surfaces, and let $\bX, \bY \in C^\infty_0([0,1]^2, \V)$ be arbitrary representatives.
Suppose $\bdy_r \bX \sim_{\thinhom} \bdy_l \bY$, and let $\homotopy: [0,1]^2 \to \R^d$ be a thin homotopy between $\bdy_r \bX$ and $\bdy_l \bY$. We define
    \begin{align} \label{eq:thinhom_hcomp}
        [\bX] \concat_h [\bY] \coloneqq [\bX \concat_h \homotopy \concat_h \bY],
    \end{align}
    where these surfaces are composable by the definition of the thin homotopy.
    \begin{figure}[!h]
        \includegraphics[width=\linewidth]{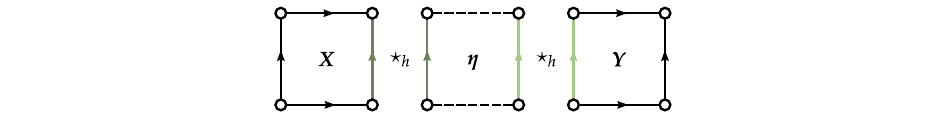}
    \end{figure}  
Vertical compositions are defined analogously. 
The fact that these compositions are well-defined is proved in Lemma 2.12 of~\cite{martins_surface_2011}. Furthermore, this satisfies the interchange law. Further details of the double group structure is found in~\Cref{apx:thin_dg}.

\subsection*{Morphism of Double Groups}
Now that we have identified the compositional structure that our surface representation has to respect as that of a double group, it is clear that our desired representation needs to be a morphism between double groups.
\begin{definition} \label{def:dg_morphism}
    Let $\dgcm, \bsE$ be double groups. A \emph{morphism of double groups} $\bF = (\bF_1, \bF_2): \dgcm \to \bsE$ is a group homomorphism $\bF_1: \sD_1 \to \sE_1$, and a function $\bF_2: \sD_2 \to \sE_2$ which preserves
    \begin{itemize}
        \item boundaries: $\partial \circ \bF_2 = \bF_1 \circ \partial$ for all $\partial \in \{\partial_l, \partial_r, \partial_b, \partial_u\}$;
        \item compositions: $\bF_2(X \concat_h Y) = \bF_2(X) \concat_h \bF_2(Y)$ and $\bF_2(X \concat_v Y) = \bF_2(X) \concat_v \bF_2(Y)$;
        \item identities: for $x \in \sD_1$, we have $\bF_2(1_x^h) = 1^h_{\bF_1(x)}$ and $\bF_2(1_x^v) = 1^v_{\bF_1(x)}$.
    \end{itemize}
\end{definition}
To sum up, we are looking for a double group morphism, 
\begin{equation}\label{eq:dg functor}
\bF^\bomega: \thindg(\V) \to \dgcm,
\end{equation}
from the domain $\thindg(\V)$ consisting of tuples of thin-homotopy classes surfaces and corresponding boundaries, into a double group $\dgcm$.

In the 1D case, such a morphism was given by parallel transport along a translation invariant connection. Thus, it is not surprising that the above double group morphism can be defined via \emph{surface holonomy}, a generalization of parallel transport, initially developed in mathematical physics~\cite{baez_higher_2004,baez_higher_2006,schreiber_smooth_2011,martins_two-dimensional_2010}; we will use the formulation from~\cite{martins_surface_2011}.
In the next section we describe a specific choice for $\dgcm$ that acts as a generalization of the general linear group that is used in the 1D case.
Computations in this group amount to simple matrix operations.
After that, we revisit the existence and uniqueness of the double group functor \eqref{eq:dg functor} defined via surface holonomy.

\section{Matrix Crossed Modules} \label{sec:matrix_crossed modules}
We will construct a double group made up of matrices.
These matrices represent automorphisms of
Baez-Crans 2-vector spaces~\cite{baez_higher-dimensional_algebras_2004}.
In order to do so, it becomes much more convenient to consider \emph{crossed modules}~\cite{brown_nonabelian_2011}, an algebraic structure which is equivalent to double groups, but is simpler to manipulate symbolically. To motivate this construction, we begin with an example of performing computations in simple double group. 
\subsection{From Groups to Double Groups.} \label{ssec:trivial_dg}
Here, we construct a double group $\dgcm(G)$ from an ordinary group $G$, where we define the edges and squares by
\begin{align} \label{eq:trivial_dg}
    \sD_1(G) \coloneqq G \quad \andd \quad \sD_2(G) \coloneqq \{ (x,y,z,w, E) \in G^5 \, :  \, E = xyz^{-1} w^{-1}\}.%
\end{align}
We will call this the \emph{trivial double group associated to $G$}, where the boundary map is the identity, $\cmb= \id$.
The fifth component, $E$, of $\dgcm_2(G)$ is the term which represents the interior, and is simply taken to be the product of the boundary terms in the counter-clockwise order. 
Now, let's consider the horizontal and vertical composition of two such squares
\[
    S= (x,y,z,w, E), \quad S' = (x',y',z',w', E') \in \dgcm_2(G).
\]

\begin{figure}[!h]
    \includegraphics[width=\linewidth]{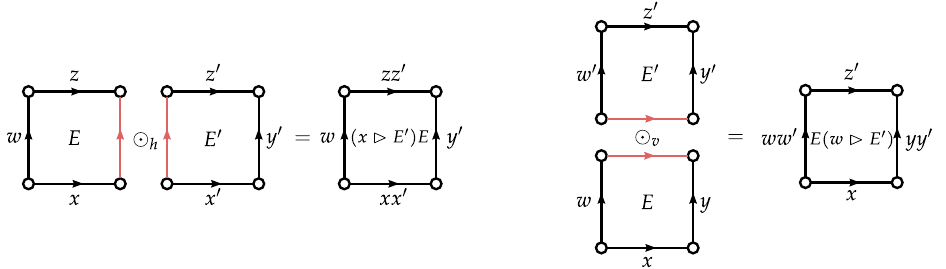}
\end{figure}

If $S$ and $S'$ are horizontally composable, meaning $y = w'$, then we can explicitly compute the boundary of the composed square as
\[
    x x' y' (z')^{-1} z^{-1} w^{-1} = x  \underbrace{\big(x' y' (z')^{-1} (w')^{-1}\big)}_{E'}  x^{-1}  \underbrace{\big(x y z^{-1} w^{-1}\big)}_{E} = (x \gt E')  E,
\]
where $x \gt y \coloneqq x y x^{-1}$ denotes the conjugation action. Therefore, the horizontal composition of $S$ and $S'$ is defined to be
\[
    S \dgcomp_h S' = (xx', y', zz', w, (x \gt E')  E).
\]
Similarly, if $S$ and $S'$ are vertically composable, meaning $z = x'$, then once again we compute the boundary of the vertically composed square as
\[
    x y y' (z')^{-1} (w')^{-1} w^{-1} = \underbrace{\big (x y z^{-1} w^{-1}\big)}_{E} w \underbrace{\big(x' y' (z')^{-1} (w')^{-1}\big)}_{E'} w^{-1} = E (w \gt E').
\]
Thus, the vertical composition of $S$ and $S'$ is defined to be
\[
    S \dgcomp_v S' = (x, yy', z, ww', E (w \gt E')).
\]
In this setting, the identities and inverses are also straightforward to define. In particular, given $x\in G$, we define the horizontal and vertical identities to be
\[
    1^h_x = (e, x, e, x, e) \quad \text{and} \quad 1^v_x = (x, e, x, e, e),
\]
where $e \in G$ is the identity. For $S = (x,y,z,w, E)$, the horizontal and vertical inverses are
\[
    S^{-h} \coloneqq (x^{-1}, w, z^{-1}, y; E^{-h}) \quad \text{and} \quad S^{-v} \coloneqq (z, y^{-1}, x, w^{-1}, E^{-v}),
\]
where
\begin{align*}
    E^{-h} = x^{-1} w z y^{-1}  = x^{-1} \gt E^{-1} \quad \andd \quad E^{-v} = z y^{-1} x^{-1} w  = w^{-1} \gt E^{-1}.
\end{align*}
One can easily check that these indeed satisfy the horizontal and vertical inverse conditions. \medskip

\subsection{Crossed Modules of Groups}
In the above example, the interior elements of squares in the double group $\dgcm(G)$ are completely determined by the boundary elements. 
However, to faithfully represent surfaces this is not sufficient, and we require a double group in which the interior element $E$ contains additional information. 
One can achieve this more general setting, by a pair of groups $(\cmG_0, \cmG_1)$ which satisfy some additional structure.

\begin{definition}
    \label{def:GCM}
        A \emph{crossed module of groups},
        \begin{align*}
            \cmG = \left(\cmb: \cmG_1\rightarrow \cmG_0,\quad \gt: \cmG_0 \rightarrow \Aut(\cmG_1) \right)
        \end{align*}
        is given by two groups $(\cmG_0, \cdot), (\cmG_1, *)$, a group morphism $\cmb: \cmG_1 \rightarrow \cmG_0$ and a left action of $\cmG_0$ on $\cmG_1$ (denoted element-wise by $g \gt : \cmG_1 \rightarrow \cmG_1$ for $g \in \cmG_0$) which is a group morphism, such that
        \begin{enumerate}
            \item[(CM1)] \textbf{(First Peiffer relation)} $\cmb(g \gt h) = g \cdot \cmb(h) \cdot g^{-1}$ \text{ for $g \in \cmG_0$ and $h \in \cmG_1$}
            \item[(CM2)] \textbf{(Second Peiffer relation)} $\cmb(h_1) \gt (h_2) = h_1 * h_2 * h_1^{-1}$ \text{ for $h_1, h_2 \in \cmG_1$}.
        \end{enumerate}
        A \emph{crossed module of Lie groups} is the same as above, except $\cmG_0$ and $\cmG_1$ are Lie groups, and all morphisms are smooth. 
\end{definition}

\begin{example}\label{ex:trivial_cm}
Let $G$ be a group. As a trivial example of a crossed module of groups, we consider $\cmG_0 = \cmG_1 = G$, the morphism $\cmb = \id : G \to G$ to be the identity, action to be conjugation.
\end{example}
\subsection*{From Crossed Modules to Double Groups}
Given a crossed module $\cmG = \left(\cmb: \cmG_1\rightarrow \cmG_0,\gt\right)$, we construct a double group $\dgcm (\cmG)$ which generalizes the construction from~\Cref{ssec:trivial_dg}.
We define the edges and squares by
\begin{align} \label{eq:double_group_of_crossed_module}
    \sD_1(\cmG) = \cmG_0 \quad \text{and} \quad \sD_2(\cmG) \coloneqq \{(x,y,z,w,E) \in \cmG_0^4 \times \cmG_1 \, : \, \cmb(E) = x y z^{-1} w^{-1}\}.
\end{align}
The compatibility condition for $\sD_2(\cmG)$ states that the interior element of squares must coincide with the boundary of the square under the map $\cmb$; thus we call $\cmb$ the \emph{crossed module boundary}. We define the composition of squares $S = (x,y,z,w,E)$ and $S' = (x', y', z', w', E')$ by,
\begin{align} \label{eq:dg_composition_hv}
    S \dgcomp_h S' = (x x', y', z z', w, (x \gt E') * E) \quad \text{and} \quad S \dgcomp_v S' = (x, y y', z', w w', E * (w \gt E')),
\end{align}
where we assume that $y=w'$ for horizontal composition and $z=x'$ for vertical composition, analogous to~\Cref{ssec:trivial_dg}. For these compositions to be well-defined, we must ensure that
\[
    \cmb((x \gt E') * E) = x \cdot x' \cdot y' \cdot (z')^{-1} \cdot z^{-1} \cdot w \quad \text{and} \quad \cmb(E * (w \gt E')) = x \cdot y \cdot y' \cdot z^{-1} \cdot w^{-1} \cdot w^{-1},
\]  
both of which can be checked by direct computation by using the first Peiffer relation. Furthermore, we must also ensure that the interchange law holds, which is once again verified by direct computation using the second Peiffer relation (see~\cite[Section 6.6]{brown_nonabelian_2011}). \medskip

For $x \in \cmG_0$, we define the horizontal and vertical identities to be
\begin{align}
    1^h_x = (e_1, x, e_1, x, e_2) \quad \text{and} \quad 1^v_x = (x, e_1, x, e_1, e_2),
\end{align}
where $e_1 \in \cmG_0$ and $e_2 \in \cmG_1$ are the respective identities. Finally, given $S = (x,y,z,w,E)$, the horizontal and vertical inverses are given by
\begin{align}
    S^{-h} \coloneqq (x^{-1}, w, z^{-1}, y, x^{-1} \gt E^{-1}) \quad \text{and} \quad S^{-v} \coloneqq (z, y^{-1}, x, w^{-1}, w^{-1} \gt E^{-1}),
\end{align}
where the inverses are understood to be the group inverses in $\cmG_0$ and $\cmG_1$.

\begin{proposition}{\cite[Section 6.6]{brown_nonabelian_2011}} \label{prop:dgcm}
    Given a crossed module $\cmG$, the structure $\dgcm(\cmG)$ is a double group.
\end{proposition}

\subsection{Crossed Modules of Matrix Lie Groups} \label{ssec:cm_matrix_lie_groups}
We now carry out the above construction of a double group by defining a crossed module
\begin{align} \label{eq:general_linear_crossed_module}
    \cmGL^{n,m,p}= (\partial: \cmGL_1^{n,m,p} \rightarrow \cmGL_0^{n,m,p}, \gt)
\end{align}
of Lie groups. 
This can be seen as the higher analogue of the classical matrix groups.
The ideas is to use automorphisms of Baez-Crans 2-vector spaces, first defined in~\cite{baez_higher-dimensional_algebras_2004}. These notions of higher general linear groups have previously been studied in~\cite{forrester-barker_representations_2003,hereida_representation_2016,martins_lie_2011,del_hoyo_general_2019}.

\begin{definition} \label{def:2_vector_space}
    A \emph{2-vector space} $\cV$ consists of two vector spaces $V_0, V_1$, along with a linear map $\phi$,
    \[
        \cV = V_1 \xrightarrow{\phi} V_0.
    \] 
\end{definition}

Equation \Cref{eq:general_linear_crossed_module} is the \emph{general linear crossed module} of the 2-vector space
\[
    \cV :  \R^{n+m} \xrightarrow{\phi}  \R^{n+p},
\]
where we choose bases for $\R^{n+m}$ and $\R^{n+p}$ such that the linear map $\phi$ has the block form
\begin{align} \label{eq:phi}
    \phi = \pmat{I & 0 \\ 0 & 0}, \quad \text{where $I\in \Mat_{n,n}$ is the identity matrix}.
\end{align}

\textbf{The Group $\cmGL_0^{n,m,p}$.} The group $\cmGL_0^{n,m,p}$ consists of invertible chain maps of $\cV$. In particular, it consists of $(F,G)$, where $F \in \GL^{n+m}$ and $G \in \GL^{n+p}$ such that $\phi F =  G \phi$. Representing $F$ and $G$ in block matrix form,
\[
    F = \pmat{A & B \\ C & D} , \quad G = \pmat{ J & K \\ L & M},
\]
the chain map condition implies that $A = J$, $B=0$ and $L=0$. Therefore, we define $\cmGL_0^{n,m,p}$ to be
\begin{equation*}
    \cmGL_0^{n,m,p} = \left\{ F,G = \pmat{A & 0 \\ B &C}, \pmat{A & D \\ 0 & E }  :  A \in \GL^n,  B \in \Mat_{m,n},  C \in \GL^m, D \in \Mat_{n,p}, E \in \GL^p\right\}.
\end{equation*}
The group multiplication of $\cmGL_0^{n,m,p}$ is matrix multiplication in each coordinate. \medskip

\textbf{The Group $\cmGL_1^{n,m,p}$.} The group $\cmGL_1^{n,m,p}$ consists of chain homotopies in $\cV$ from the identity to an element of $\cmGL_0^{n,m,p}$. In particular, it consists of linear maps $H: \R^{n+p} \rightarrow  \R^{n+m}$ such that
\begin{align}
    H\phi = F- I \quad \andd \quad \phi H = G - I
\end{align}
for some $(F,G) \in \cmGL_0^{n,m,p}$. Expressing these conditions in block matrix form, we define
\begin{equation}
    \label{eq:GL1_def}
    \cmGL_1^{n,m,p} \coloneqq \left\{ h = \pmat{A - I & D \\ B & U} \, : \, A \in \GL^n, \, B \in \Mat_{m,n}, \, D \in \Mat_{n,p}, \, H \in \Mat_{m,p} \right\}.
\end{equation}

\textbf{Group Structure of $\cmGL_1^{n,m,p}$.} Group multiplication in $\cmGL_1^{n,m,p}$ corresponds to horizontal composition of chain homotopies. Concretely, the group multiplication is
\begin{align} \label{eq:cmgl2_multiplication}
    H*H' = H + H' + H\phi H'.
\end{align}
The unit is the zero map, which we simply denote by $0$, and the inverse of $H$ with respect to $*$ is
\begin{align}
    H^{-*} = -(I + H\phi)^{-1} H = -H(I + \phi H)^{-1}.
\end{align}

\textbf{Crossed Module Boundary.} The crossed module boundary map $\cmb: \cmGL_1^{n,m,p} \rightarrow \cmGL_0^{n,m,p}$ sends a homotopy to its target. This is defined as
\begin{equation} \label{eq:matrix_cmb}
    \cmb(H) = (H\phi + I, \phi H + I).
\end{equation}

\textbf{Crossed Module Action.} The last part of the crossed module structure to define is the action $\gt$ of $\cmGL_0^{n,m,p}$ on $\cmGL_1^{n,m,p}$. Given  $H \in \cmGL_1^{n,m,p}$ and $(F,G) \in \cmGL_0^{n,m,p}$, the action is defined by
\begin{equation}
    (F,G) \gt H \coloneqq FHG^{-1}.
\end{equation}

\begin{theorem}{\cite{martins_lie_2011}}
    The structure
    \[
    \cmGL^{n,m,p} = \big( \partial: \cmGL_1^{n,m,p} \rightarrow \cmGL_0^{n,m,p}, \gt)
    \]
    defined in this section is a crossed module of Lie groups.
\end{theorem}

\subsection{Crossed Modules of Matrix Lie Algebras} \label{ssec:cm_matrix_lie_algebras}

Crossed modules of Lie algebras (also called differential crossed modules), are the infinitesimal version of crossed modules of Lie groups.

\begin{definition}
    \label{def:DCM}
        A \emph{crossed module of Lie algebras},
        \begin{align*}
            \cmg = \left( \delta: \cmg_1 \rightarrow \cmg_0,\quad \gt: \cmg_0 \rightarrow \Der(\cmg_1)\right),
        \end{align*}
        is given by Lie algebras $(\cmg_0,[\cdot,\cdot]_1)$ and $(\cmg_1, [\cdot, \cdot]_2)$, a Lie algebra morphism $\delta: \cmg_1 \rightarrow \cmg_0$ and an action of $\cmg_0$ on $\cmg_1$ (denoted element-wise by $X \gt : \cmg_1 \rightarrow \cmg_1$ for $X \in \cmg_0$), which is a morphism of Lie algebras, such that
        \begin{enumerate}
            \item[(DCM1)] \textbf{(First Peiffer relation)} $\delta(X \gt v) = [X, \delta(v)]_{1}$; \text{for $X \in \cmg_0$ and $v \in \cmg_1$}.
            \item[(DCM2)] \textbf{(Second Peiffer relation)} $\delta(u)\gt (v) = [u,v]_{2}$; \text{for $u, v \in \cmg_1$}.
        \end{enumerate}
\end{definition}

For a crossed module of Lie groups, $\cmG = (\cmb: \cmG_1 \to \cmG_0)$, one can construct its associated crossed module of Lie algebras by considering the Lie algebras of the two Lie groups, along with the induced boundary map and action~\cite{faria_martins_crossed_2016}.
We apply this to $\cmGL^{n,m,p} \label{pg:cmgl}$ to get
\[
    \cmgl^{n,m,p} = \big( \cmb: \cmgl_1^{n,m,p} \rightarrow \cmgl_0^{n,m,p}, \gt).
\]
Because $\cmGL_0^{n,m,p}$ just contains pairs of matrices with matrix multiplication as the group multiplication, the Lie algebra $\cmgl_0^{n,m,p}$ is simply the space of all chain maps,
\[
    \cmgl_0^{n,m,p} = \left\{ X,Y = \pmat{A & 0 \\ B &C}, \pmat{A & D \\ 0 & E } : A \in \fgl^n,  B \in \Mat_{m,n},  C \in \fgl^m, D \in \Mat_{n,p}, E \in \fgl^p\right\}.
\]
The Lie bracket is the usual commutator of matrices, denoted $[\cdot, \cdot]$. 
Next, the Lie algebra of $\cmGL_1^{n,m,p}$ consists of linear transformations $Z: \R^{n+p} \rightarrow  \R^{n+m}$,
\[
    \cmgl_1^{n,m,p} \coloneqq \left\{ Z = \pmat{R & S \\ T & U} : R \in \Mat_{n,n}, \, S \in \Mat_{m,n}, \, T \in \Mat_{n,p}, \, U \in \Mat_{m,p}\right\},
\]
where the Lie bracket is the commutator with respect to the $*$-product,
\begin{align} \label{eq:star_commutator}
    [Z, Z']_* \coloneqq Z\phi Z' - Z' \phi Z.   
\end{align}
The boundary map, $\delta: \cmgl_1^{n,m,p} \rightarrow \cmgl_0^{n,m,p}$ and the action of $\cmgl_0^{n,m,p}$ on $\cmgl_1^{n,m,p}$ are given by
\[
    \cmb(Z) = (Z \phi, \phi Z) \quad \andd \quad (X,Y) \gt Z = X Z - Z Y.
\]

Finally, note the following relationship between the Lie group and Lie algebra actions: recall that for any action of Lie groups $\gt : G \to \Aut(H)$, there is a corresponding action of the Lie group $G$ on the Lie algebra $\fh$. 
In the case of the general linear crossed modules, the action $\gt: \cmGL_0^{n,m,p} \to \Aut(\cmgl_1^{n,m,p})$ is defined by
\begin{align} \label{eq:lg_la_action}
    (F,G) \gt X = F X G^{-1}.
\end{align}

\section{Matrix Surface Development}\label{sec:surface_development}
We are now in a position to explicitly define surface development. 
The surface development is realized as a translation-invariant surface holonomy, a generalization of parallel transport initially developed in mathematical physics~\cite{baez_higher_2004,baez_higher_2006,schreiber_smooth_2011,martins_two-dimensional_2010}; we will use the formulation from~\cite{martins_surface_2011}.

\begin{definition} \label{def:2_connection}
    A \emph{(fake-flat, translation-invariant) 2-connection} $\bomega = (\cona, \conc)$ over $\R^d$ valued in the differential crossed module $\cmg  = (\cmb: \cmg_1 \to \cmg_0)$ consists of a linear differential 1-form $\cona \in \Lin(\V, \cmg_0)$ and a linear differential 2-form $\conc \in \Lin(\Lambda^2\V, \cmg_1)$ of the form
    \begin{align}
        \cona \coloneqq \sum_{i=1}^d \cona^i \, dx_i \quad \text{and} \quad \conc \coloneqq \sum_{i < j} \conc^{i,j} \, dx_i \wedge dx_j,
    \end{align}
    where $\cona^i \in \cmg_0$ and $\conc^{i,j} \in \cmg_1$, which satisfies $\cmb \conc = \frac12[\cona, \cona]$ (the fake-flatness condition).
\end{definition}

Given such a 2-connection $\bomega = (\alpha, \gamma)$, we are ready to define our desired representation as a morphism of double groups (\Cref{def:dg_morphism}), %
\begin{align} \label{eq:sd_functor}
    \bF^{\bomega}: \thindg(\V) \to \dgcm(\cmG),
\end{align}
Here, $\dgcm(\cmG)$ is the double group associated to a crossed module of Lie groups $\cmG$, given in~\Cref{eq:double_group_of_crossed_module}. 
The morphism $\bF^\bomega$ is itself a tuple $(\bF^\bomega_1, \bF^\bomega_2)$.
This consists of a \emph{path component} $\bF^\bomega_1$, which is a group homomorphism defined by path development, $\bF_1^\bomega = F^\alpha: \thinpath(\V) \to \sD_1(\cmG) = \cmG_0$. 
However, it also contains the \emph{surface component}, $\bF_2^\bomega : \thinsurface(\V) \to \sD_2(\cmG)$, defined on a surface $\bX \in \thinsurface(\V)$ with boundary paths $\bx, \by, \bz, \bw$ by
\begin{align}
    \bF_2^\bomega(\bX) = \big(F^\alpha(\bx), F^\alpha(\by), F^\alpha(\bz), F^\alpha(\bw), H^\bomega(\bX)\big) \in \sD_2(\cmG).
\end{align}
The map $H^\bomega : \thinsurface(\V) \to \cmG_1$ is called the \emph{surface development}\footnote{The term \emph{surface development} will refer to $H^\bomega$, while we will refer to $\bF^\omega$ as the \emph{surface development functor}. While we primarily work with $H^\bomega$, we note that the compositions for $H^\bomega$ require information about the paths in $\bF^\bomega$.}, and is also defined as the solution to a differential equation.

\begin{definition}{\cite[Equation 2.13]{martins_surface_2011}} \label{def:sh}
    Let $\cmG = (\cmb: \cmG_1 \to \cmG_0, \gt)$ be a crossed module, and let $\bomega = (\cona, \conc)$ be a 2-connection valued in $\cmg$. Let $\bX \in C^\infty([0,1]^2, \V)$. The \emph{surface development of $\bX$ with respect to $\bomega$} is the solution of the differential equation for $H^{\bomega}_{s,t}(\bX) : [0,1]^2 \to \cmG_1$, defined by
    \begin{align}\label{eq:general_sh}
        \frac{\partial H^{\cona, \conc}_{s,t}(\bX)}{\partial t} = dL_{H^{\cona, \conc}_{s,t}(\bX)} \int_0^s F^{\cona}(\bx^{s',t}) \gt \conc\left( \pd{\bX_{s',t}}{s}, \pd{\bX_{s',t}}{t}\right) ds', \quad H^{\cona, \conc}_{s,0}(\bX) = e_{\cmG_1},
    \end{align}
    where $\gt$ is the induced action of $\cmG_0$ on $\cmg_1$, and $\bx^{s,t}: [0,1] \to \V$ is is the \emph{tail path} defined by
\begin{align} \label{eq:tail_path}
    \bx^{s,t}_u \coloneqq \left\{ \begin{array}{cl}
        \bX_{0,2ut} & : u \in [0,1/2] \\
        \bX_{(2u-1)s, t } & : u \in (1/2, 1]
    \end{array} \right.
\end{align}

\noindent We define
    \[
        H^{\cona, \conc}(\bX) \coloneqq H^{\cona, \conc}_{1,1}(\bX).
    \]
\end{definition}
This yields our desired morphism. In fact,~\cite{schreiber_smooth_2011} shows that all smooth morphisms arise from surface holonomy with respect to a (not necessarily translation-invariant) 2-connection. 

\begin{theorem}{\cite[Theorem 2.32]{martins_surface_2011}}
    Let $\bomega$ be a 2-connection valued in $\cmg$. The maps $\bF^\bomega = (\bF_1^\bomega, \bF_2^\bomega): \thinsurface(\V) \to \dgcm(\cmG)$ defined in~\Cref{eq:sd_functor} is a morphism of double groups. 
\end{theorem}

This result has three important components. First, it is invariant under thin homotopy of surfaces~\cite[Corollary 2.31]{martins_surface_2011}. Second, it is well-defined as an element of the double group; in particular, it satisfies the \emph{nonabelian Stokes' theorem}:
\begin{align}
    \cmb H^{\bomega}(\bX) = F^\cona(\partial \bX),
\end{align}
where $\cmb: \cmG_1 \to \cmG_0$ is the crossed module boundary, and $\partial\bX: [0,1] \to \V$ is the boundary path.%

\begin{remark} \label{ex:trivial_nonabelian_stokes}
    In the case of the trivial crossed module $\cmG$ from~\Cref{ex:trivial_cm}, the associdated double group is the trivial double group $\dgcm(G)$ defined in~\Cref{eq:trivial_dg}.
    Choose a 1-connection $\cona \in \Lin(\V, \fg)$, and define the 2-connection as the curvature $\conc = [\cona,\cona]\in \Lin(\Lambda^2\V, \fg)$ of $\cona$. In this setting, surface development simply recovers the path development of the boundary. This has previously been studied as a nonabelian Stokes' theorem (see~\cite{karp_product_1999}), going back to the work of Schlesinger~\cite{schlesinger_parallelverschiebung_1928}.
\end{remark}

Third, the maps preserve compositions and identities on thin homotopy classes; in particular,
\begin{align}
    H(\bX \comp_h \bY) = [F(\bx) \gt H(\bY)] * H(\bX)   \quad \andd \quad H(\bX \comp_v \bY) = H(\bX) * [ F(\bw) \gt H(\bY)],
\end{align}
and $H^\bomega(1^h_\bx) = H^\bomega(1^v_\bx) = e_1$, where $\bX,\bY \in \thinsurface(\V)$ are appropriately composable (thin homotopy classes of) surfaces, $\bx \in \thinpath(\V)$, and $e_1 \in \cmG_1$ is the identity.

\subsection{Matrix Surface Development} \label{ssec:matrix_sh}
So far, we have considered surface development for 2-connections valued in arbitrary differential crossed modules. Here, we will specialize to the case of \emph{matrix surface development}, when the 2-connections are valued in $\cmgl^{n,m,p}$.\medskip

\textbf{1-Connections and Path Development.}
We denote a 1-connection valued in $\cmgl_0^{n,m,p}$ by $(\cona, \conb) \in \Lin(\V, \cmgl_0^{n,m,p}) \cong \Lin(\V, \fgl^{n+m}) \op \Lin(\V, \fgl^{n+p})$, where
\[
    (\cona, \conb) = \sum_{i=1}^d (\cona^i, \conb^i) dx_i \quad \text{with} \quad \cona^i, \conb^i = \pmat{A^i & 0 \\ B^i & C^i}, \pmat{A^i & D^i \\ 0 & E^i} \in \cmgl_0^{n,m,p}.
\]
We view each component of $(\cona, \conb)$ separately as $\cona \in \Lin(\V, \fgl^{n+m})$ and $\conb \in \Lin(\V, \fgl^{n+p})$. Then, given a smooth path $\bx: [0,1] \to \V$, we denote the path development of $\bx$ by
\begin{align} \label{eq:matrix_holonomy_1_hol_part}
    F^{\cona, \conb}(\bx) = (F^{\cona}(\bx), F^{\conb}(\bx)) \in \cmGL_0^{n,m,p}.
\end{align}
In particular, $F^{\cona}(\bx) \in \GL^{n+m}$ is the path development with respect to $\cona$ and $F^{\conb}(\bx) \in \GL^{n+p}$ is the path development with respect to $\conb$.\medskip

\textbf{2-Connections.} We denote the 2-form in the 2-connection by $\conc \in \Lin(\Lambda^2\V, \cmgl_1^{n,m,p})$, where
\[
    \conc = \sum_{i < j} \conc^{i,j} dx_i \wedge dx_j \quad \text{with} \quad \conc^{i,j} = \pmat{R^{i,j} & S^{i,j} \\ T^{i,j} & U^{i,j}} \in \cmgl_1^{n,m,p}.
\]

However, there are some restrictions on the $\cona^i, \conb^i$ and $\conc^{i,j}$ due to the fake-flatness condition (\Cref{def:2_connection}). In particular, for each $i,j$, we must have $\cmb(\conc^{i,j}) = ([\cona^i, \cona^j], [\conb^i, \conb^j])$, leading to the following definition.

\begin{definition} \label{def:matrix_2connection}
    A \emph{matrix 2-connection} valued in $\cmgl^{n,m,p}$ is a triple $ \con = (\cona, \conb, \conc)$ such that $(\cona, \conb) \in \Lin(\V, \cmgl_0^{n,m,p})$ and $\conc \in \Lin(\Lambda^2\V, \cmgl_1^{n,m,p})$, where the components in block matrix form are
    \[
        \cona^i = \pmat{A^i & 0 \\ B^i & C^i},  \quad \conb^i =\pmat{A^i & D^i \\ 0 & E^i}, \quad \conc^{i,j} = \pmat{R^{i,j} & S^{i,j} \\ T^{i,j} & U^{i,j}},
    \]
    such that for all $i < j$, we have $[C^i, C^j] = [E^i, E^j] = 0$, $R^{i,j}  = [A^i, A^j]$,
    \begin{align*}
        S^{i,j}  = A^i D^j - A^j D^i + D^i E^j - D^j E^i  \quad \andd \quad T^{i,j}  = B^i A^j - B^j A^i + C^i B^j - C^j B^i.
    \end{align*}
    We let $\cM^{n,m,p}(\V)$ denote the set of matrix 2-connections on $\V$ valued in $\cmgl^{n,m,p}$. When the domain is clear, we simplify the notation and simply write $\cM^{n,m,p} \coloneqq \cM^{n,m,p}(\V)$.
\end{definition}

This implies that the choice of 1-connection fully determines the matrices $R^{i,j}, S^{i,j}$ and $T^{i,j}$, and the choice of 2-connection reduces to choosing a set of matrices $U^{i,j} \in \R^{m \times p}$. Furthermore, we note that $\cM^{n,m,p}$ is not a linear space due to the nonlinear conditions $[C^i, C^j] = 0$ and $[E^i, E^j] = 0$.\medskip

\textbf{Metrics on 2-Connections.}
While $\cM^{n,m,p}$ is not a linear space, it is a subset of the linear space $\Lin(\V, \cmgl_0^{n,m,p}) \op \Lin(\Lambda^2 \V, \cmgl_1^{n,m,p})$, which we equip with the Frobenius norm. Given $\con = (\cona, \conb, \conc) \in \Lin(\V, \cmgl_0^{n,m,p}) \op \Lin(\Lambda^2 \V, \cmgl_1^{n,m,p})$, we define
\begin{align} \label{eq:frob_norm_matrix_2_connection}
    \|\con\|_{\Fr}^2 \coloneqq \sum_{i=1}^d \|\cona^i\|_{\Fr}^2 + \sum_{i=1}^d \|\conb^i\|_{\Fr}^2 + \sum_{1 \leq i < j \leq d} \|\conc^{i,j}\|_{\Fr}^2,
\end{align}
where $\|\cdot\|_{\Fr}$ on the right side of the equation is the usual Frobenius norm for matrices $A \in \Mat_{n,m}$.
Then, we define a metric between 2-connections $d_{\cM}: \cM^{n,m,p} \times \cM^{n,m,p} \to \R$ by
\begin{align} \label{eq:2_connection_metric}
    d_\cM(\con_1, \con_2) \coloneqq \|\con_1 - \con_2\|_{\Fr}.
\end{align}

\textbf{Surface Development.}
Now, we will consider the surface development equation in~\Cref{eq:general_sh} for matrix 2-connections. We make this equation explicit by noting that left multiplication is given by $dL_H (X) = (I + H \phi)X$, and using the action $\gt$ of $\cmGL_0^{n,m,p}$ on $\cmgl_1^{n,m,p}$ from~\Cref{eq:lg_la_action}.
The \emph{matrix surface development equation} with respect to the matrix 2-connection $\con = (\cona, \conb, \conc) \in \cM^{n,m,p}$ is
\begin{align}\label{eq:matrix_sh1}
    \frac{\partial H^{\con}_{s,t}(\bX)}{\partial t} = (I + H^{\con}_{s,t}(\bX)\phi)  \int_0^s F^{\cona}(\bx^{s',t}) \cdot\conc\left( \pd{\bX_{s',t}}{s}, \pd{\bX_{s',t}}{t}\right) \cdot \left(F^{\conb}(\bx^{s',t})\right)^{-1} ds',
\end{align}
with boundary conditions $H^{\con}_{s,0}(\bX) = H^{\con}_{0,t}(\bX) = 0$. Note that we can evaluate $\conc$ to get
\[
    \conc\left( \pd{\bX_{s',t}}{s}, \pd{\bX_{s,t}}{t}\right) = \conc \left(J_{s,t}[\bX]\right) \quad \text{where} \quad J[\bX] : [0,1]^2 \to \Lambda^2 \V \quad \text{with} \quad J^{i,j}_{s,t}[\bX] \coloneqq \vmat{ \pd{\bX^i_{s,t}}{s} & \pd{\bX^i_{s,t}}{t}\\ \pd{\bX^j_{s,t}}{s} & \pd{\bX^j_{s,t}}{t}}
\]
denotes the Jacobian of $\bX$.

\section{Young Surface Development}\label{sec:Young}
We will now move beyond the smooth setting, and consider surface development in the Young regime. In particular, we will consider the generalized $\rho$-H\"older spaces, as defined in~\cite{harang_extension_2021}, in the case of $\rho > \frac12$. This is the usual definition for paths, but we must consider the regularity of 2D increments for surfaces.

\begin{remark}
    For readers primarily interested in characteristic functions for smooth random surfaces, we emphasize that the results in~\Cref{sec:parametrized} and~\Cref{sec:characteristicness} hold for smooth surfaces as a special case, and can be read independently from this section by replacing $\rho$-H\"older with smooth.
\end{remark}

\begin{definition}
    Let $\rho \in (0,1]$. We say that a path $\bx \in C([0,1], V)$ is $\rho$-H\"older, $\bx \in C^\rho([0,1], \V)$, if
    \begin{align}
        \|\bx\|_\rho \coloneqq \llb\bx\rrb_{\rho} + \|\bx\|_\infty < \infty \quad \text{where} \quad \llb \bx\rrb_{\rho} \coloneqq \sup_{s < t} \frac{|\bx_t - \bx_s|}{|t -s|^\rho} < \infty.
    \end{align}
\end{definition}

For a \emph{rectangle} $R = [s_1,s_2] \times [t_1, t_2] \subset [0,1]^2$, we denote the \emph{2D increment} of $\bX$ by
\begin{align} \label{eq:2d_increment}
    \square_R[\bX] \coloneqq \bX_{s_1, t_1} - \bX_{s_1, t_2} - \bX_{s_2, t_1} + \bX_{s_2, t_2}.
\end{align}

\begin{definition}
    Let $\rho \in (0,1]$. We define the following quantities for a surface $\bX \in C([0,1]^2, V)$.
    \begin{align*}
        |\bX|_{\rho, (1)} = \sup_{\substack{t\in [0,1]\\s_1 < s_2}} \frac{|\bX_{s_1,t} - \bX_{s_2, t}|}{|s_1 - s_2|^\rho}, \quad |\bX|_{\rho, (2)} \sup_{\substack{s\in [0,1]\\ t_1 < t_2}} \frac{|\bX_{s,t_1} - \bX_{s,t_2}|}{|t_1 - t_2|^\rho}, \quad |\bX|_{\rho, (1,2)} = \sup_{\substack{s_1 < s_2\\t_1 < t_2}} \frac{|\square_R[\bX]|}{|s_1 - s_2|^\rho |t_1 - t_2|^\rho}.
    \end{align*}
    We define $\llb \bX \rrb_{\rho} \coloneqq |\bX|_{\rho, (1)} + |\bX|_{\rho, (2)} + |\bX|_{\rho, (1,2)}$, and say that $\bX$ is $\rho$-H\"older, $\bX \in C^\rho([0,1]^2, \V) \label{pg:holder_space}$, if 
    \begin{align}
        \|\bX\|_\rho \coloneqq \llb \bX\rrb_\rho + \|\bX\|_\infty < \infty.
    \end{align}
\end{definition}

\subsection{Existence and Continuity}
Consider a matrix 2-connection $\con = (\cona, \conb, \conc) \in \cM^{n,m,p}$. Our main task is to reformulate the matrix surface development equation in~\Cref{eq:matrix_sh1} using multidimensional Young integration~\cite{harang_extension_2021}. We start by considering a smooth surface $\bX \in C^\infty([0,1]^2, \V)$, where we note that we can express~\Cref{eq:matrix_sh1} as a 1D CDE in the form of 
\begin{align} \label{eq:msh_decomp1}
    \partial_t  H^\bomega_{s,t}(\bX) =  (I + H^\bomega_{s,t}(\bX)\phi) \partial_t Z^\bomega_{s,t}(\bX), \quad H_{s,0} = 0.
\end{align}
Here, the surface $Z^\bomega_{s,t}: [0,1]^2 \to \cmgl^{n,m,p}_1$ is defined by the integral
\begin{align} \label{eq:msh_decomp2}
    Z^\bomega_{s,t}(\bX) \coloneqq   \int_0^t \int_0^s F^\alpha(\bx^{s',t'}) \cdot \conc \cdot F^\beta(\bx^{s',t'})^{-1} \, dA_{s',t'}(\bX),
\end{align}
where $A(\bX): [0,1]^2 \to \Lambda^2\V$ is the \emph{area process} of $\bX$,
\begin{align}
    A_{s,t}(\bX) = \int_0^t \int_0^s  J_{s',t'}(\bX) \,ds' dt'.
\end{align}
We claim that this formulation of the surface development equation is well-defined for surfaces in the Young regime. Let's suppose that $\bX \in C^\rho([0,1]^2, \V)$ with $\rho > \frac12$. The \emph{area process} is simply the signed area of $\bX$ restricted to $[0,s] \times [0,t]$, and can be computed by 1D Young integration of the boundary. Similarly, the path developments $F^\alpha$ and $F^\beta$ of the tail paths are well-defined Young CDEs; thus~\Cref{eq:msh_decomp2} can be computed as a 2D Young integral. Finally, by restricting $Z^\bomega_{s,t}(\bX)$ to $s = 1$, we obtain a $\rho$-H\"older driving signal, and thus~\Cref{eq:msh_decomp1} is also a Young CDE. \medskip

The main theorem of this section shows that surface development is well-defined and continuous in the Young regime. The proof primarily consists of verifying the generalized H\"older regularity of the various components, and can be found in~\Cref{apx:young_proofs}.

\begin{theorem} \label{thm:main_sd_holder}
    The surface development map $H: C^{\rho}([0,1]^2, \V) \times \cM^{n,m,p} \to \cmGL_1^{n,m,p}$ is well-defined and locally Lipschitz with respect to $C^{\rho}([0,1]^2, \V)$ and  $\cM^{n,m,p}$.
\end{theorem}

\subsection{Young Double Groups}
We will now consider a generalization of the thin double group which contains lower regularity surfaces in the Young regime. In order to take advantage of the $\rho$-H\"older continuity of surface development, we will use the closure (under the generalized H\"older norms) of smooth paths ($k=1$) and surfaces ($k=2 \label{pg:smoothcl_holder}$),
\begin{align}
    C^{0, \rho}([0,1]^k, \V) \coloneqq \left\{ \bX \in C^\rho([0,1]^k, \V) \, : \, \exists \bX^n \in C^\infty([0,1]^k, \V) \text{ such that } \bX^n \xrightarrow{\rho} \bX\right\}.
\end{align}

However, as the definition of both 1D and 2D thin homotopy require differentiability, we will define this through limits of smooth paths and surfaces.

\begin{definition} \label{def:thin_smoothcl_holder}
    We say that $\bX, \bY \in C^{0, \rho}_0([0,1]^k, \V)$ are \emph{thin homotopy equivalent}, $\bX \sim_{\thinhom} \bY$, if for any sequence $\bX^n \in C^\infty_0([0,1]^k, \V)$ such that $\bX^n \xrightarrow{\rho} \bX$, there exists a sequence $\bY^n \in C^\infty_0([0,1]^k, \V)$ such that $\bY^n \xrightarrow{\rho} \bY$ and $\bX^n \sim_{\thinhom} \bY^n$, and similarly if we start with a sequence for $\bY$. We denote the set of thin homotopy classes of paths and surfaces by
    \begin{align}
        \mathsf{T}_k^{0,\rho}(\V) \coloneqq C^{0,\rho}_0([0,1]^k, \V)/\sim_{\thinhom}.
    \end{align}
\end{definition}

\begin{remark}
    We note that this notion of thin homotopy of paths is still equivalent to tree-like equivalence for $\rho$-H\"older paths~\cite{boedihardjo_signature_2016}
\end{remark}

We can similarly define composition of thin homotopy classes.

\begin{definition} \label{def:young_composition}
    We say that $[\bX], [\bY] \in \thinsurface^{0,\rho}(\V)$ are \emph{horizontally (resp.~vertically) composable} if for every sequence $\bX^n \in C^{\infty}([0,1]^2, \V)$ such that $\bX^n \xrightarrow{\rho} \bX \in [\bX]$, there exists a sequence $\bY^n \in C^{\infty}([0,1]^2, \V)$ such that $\bY^n \xrightarrow{\rho} \bY \in [\bY]$ and $\bX^n$ and $\bY^n$ are horizontally (resp.~vertically) composable. We define their compositions respectively as
    \begin{align}
        [\bX] \concat_h [\bY] = \left[\lim_{n \to \infty} \bX^n \concat_h \bY^n\right] \quad \andd \quad [\bX] \concat_v [\bY] = \left[ \lim_{n \to \infty} \bX^n \concat_v \bY^n \right].
    \end{align}
\end{definition}

Identities and inverses are defined in the same way as the smooth setting (see~\Cref{apx:thin_dg}).

\begin{proposition} \label{prop:thin_holder_dg}
    Let $\rho > \frac12$. The pair $\thindg^{0,\rho}(\V) \coloneqq (\thinpath^{0,\rho}(\V), \thinsurface^{0,\rho}(\V))$ equipped with compositions in~\Cref{def:young_composition}, along with identities and inverses defined in~\Cref{apx:thin_dg}, is a double group.
\end{proposition}
\begin{proof}
    We begin by proving that the four boundary maps $\partial_i: \thinsurface^{0,\rho}(\V) \to \thinpath^{0,\rho}(\V)$ are well defined. Suppose $\bX, \bY \in C^{0,\rho}([0,1]^2, \V)$ such that $\bX \sim_{\thinhom} \bY$, so that there exists $\bX^n, \bY^n \in C^\infty([0,1]^2, \V)$ such that $\bX^n \sim_{\thinhom} \bY^n$ and $\bX^n, \bY^n \xrightarrow{\rho} \bX, \bY$. Therefore, for each boundary path, we have $\partial_i \bX^n \sim_{\thinhom} \partial_i \bY^n$ and $\partial_i \bX^n, \partial_i \bY^n \xrightarrow{\rho} \partial_i \bX, \partial_i \bY$, so that $\partial_i \bX \sim_{\thinhom} \partial_i \bY$.
    The remaining properties (associativity and inverses) can be directly verified by considering smooth approximations and using the analogous properties in the smooth setting.
\end{proof}

Then, an immediate corollary of~\Cref{thm:main_sd_holder} shows that surface development is a morphism of double groups in the Young regime.

\begin{theorem} \label{thm:young_sd_functorial}
    Let $\bomega \in \cM^{n,m,p}$ be a matrix 2-connection. The maps $\bF^\bomega = (\bF_1, \bF_2): \thindg^{0,\rho}(\V) \to \dgcm(\cmGL^{n,m,p})$, where $\bF^\bomega_1 = F^{\alpha, \beta} : \thinpath^{0,\rho}(\V) \to \cmGL_0^{n,m,p}$ is path development in the Young regime, and
    \begin{align}
    \bF_2^\bomega(\bX) = \big(F^\alpha(\bx), F^\alpha(\by), F^\alpha(\bz), F^\alpha(\bw), H^\bomega(\bX)\big) \in \sD_2(\cmG),
    \end{align}
    where $H^\bomega: \thinsurface^{0,\rho}(\V) \to \cmGL_1^{n,m,p}$ is the Young surface development defined in~\Cref{thm:main_sd_holder}, is well-defined and is a morphism of double groups. 
\end{theorem}

\section{Parametrized Surface Development} \label{sec:parametrized}

In this section, we show that matrix surface development separates surfaces (up to translation) for parametrized surfaces in the Young regime. The results in this section do not require considering the closure of smooth surfaces. For $\bX \in C^{\rho}_0([0,1]^2, \V)$, we define
\begin{align} \label{eq:parametrized_surface}
    \obX_{s,t} \coloneqq (s,t, \bX_{s,t}) : [0,1]^2 \to \pV
\end{align}
to be the \emph{parametrized surface of $\bX$}.
We denote 2-connections for parametrized surfaces in $\pV$ by
\begin{align*}
    \cona = \cona^s ds + \cona^t dt + \sum_{i=1}^d \cona^i dx_i, \quad  \conc = \conc^{s,t} ds \wedge dt + \sum_{i=1}^d \big(\conc^{s,i} ds \wedge dx_i + \conc^{t,i} dt \wedge dx_i\big) + \sum_{i < j} \conc^{i,j} dx_i \wedge dx_j,
\end{align*}
where we use the $s$ and $t$ superscripts as indices for the $s$ and $t$ coordinates in $\obX$. We prove the separation property by considering the boundary (\Cref{prop:injectivity_boundary}) and interior (\Cref{prop:injectivity_interior}) of a surface $\bX$ individually. We obtain the following result by putting these two together. 

\begin{theorem} \label{thm:sh_injectivity}
    Let $\bX, \bY \in C^{\rho}_0([0,1]^2, \V)$ with $\bX \neq \bY$. There exists $\con \in \cM^{n,m,p}(\pV)$ such that
    \[
        H^{\con}(\obX) \neq H^{\con}(\obY).
    \]
\end{theorem}

\subsection{Characterizing the Boundary}

To capture information about the boundary, we consider $\cmgl^{n,0,0}$ such that in the block matrix notation of~\Cref{def:matrix_2connection}, we only have a nontrivial upper left block. Given the fake-flatness conditions on $\con = (\cona, \conb, \conc)\in \cM^{n,0,0}(\R^{d+2})$, we must have $\conb = \cona$ and $\conc = \kappa^\alpha = \frac12[\cona,\cona]$. This reduces surface development to the case discussed in~\Cref{ex:trivial_nonabelian_stokes}, where $H^{\con}(\bX) = F^{\cona}(\partial \bX)$.
We begin with a preliminary lemma. 

\begin{lemma} \label{lem:tleq_parametrized_boundary}
    Let $\bX, \bY \in C^{\rho}_{\zor}([0,1]^2, \V)$. Then, 
    $\partial \obX \sim_{\tleq} \partial \obY$ if and only if $\partial \bX = \partial \bY$. 
\end{lemma}
\begin{proof}
    If $\partial \bX = \partial \bY$, then $\partial \obX = \partial \obY$, and thus $\partial \obX \sim_{\tleq} \partial \obY$. Now, suppose $\partial \obX \sim_{\tleq} \partial \obY$. Let $\obu^{\bX}, \obv^{\bX}, \obw^{\bX}, \obz^{\bX} : [0,1] \to \pV$ denote the bottom, right, upper, and left boundary paths of $\obX$, and similarly for $\obY$. 
    By definition, $\partial \obX \sim_{\tleq} \partial \obY$, means that
    \[  
        \partial \obX \concat (\partial \obY)^{-1} = \obu^{\bX} \concat \obv^{\bX} \concat (\obw^{\bX})^{-1} \concat (\obz^{\bX})^{-1} \concat \obz^{\bY} \concat \obw^{\bY} \concat (\obv^{\bY})^{-1} \concat (\obu^{\bX})^{-1}
    \]
    is tree-like equivalent to the constant path. This implies that
    \begin{align} \label{eq:tleq_boundary}
        \obu^{\bX} \sim_{\tleq} \obu^{\bY}, \quad \obv^{\bX} \sim_{\tleq} \obv^{\bY}, \quad \obw^{\bX} \sim_{\tleq} \obw^{\bY} \quad \andd \quad \obz^{\bX} \sim_{\tleq} \obz^{\bY},
    \end{align}
    which is due to the $s$ and $t$ parameters in $\obX$ and $\obY$. 
    Because each of the boundary paths has a monotone coordinate,~\Cref{eq:tleq_boundary} implies that 
    \[
        \obu^{\bX} = \obu^{\bY}, \quad \obv^{\bX} = \obv^{\bY}, \quad \obw^{\bX} = \obw^{\bY} \quad \andd \quad \obz^{\bX} = \obz^{\bY}.
    \]
\end{proof}

\begin{proposition} \label{prop:injectivity_boundary}
    Let $\bX, \bY \in C^{\rho}_{\zor}([0,1]^2,\V)$. If $\partial \bX \neq \partial \bY$, there exists $\con \in \cM^{n,0,0}(\R^{d+2})$ such that
    \[
        H^\con(\obX) \neq H^\con(\obY).
    \]
\end{proposition}
\begin{proof}
    We consider a matrix 2-connection $\con = (\cona, \cona, \kappa^\cona) \in \cM^{n,0,0}(\R^{d+2})$ where $H^{\con}(\obX) = F^{\cona}(\obX)$. Thus, it suffices to show there exists a 1-connection $\cona$ such that $F^{\cona}(\partial \obX) \neq F^{\cona}(\partial \obY)$. Because path holonomies separate paths up to tree-like equivalence by~\cite[Theorem 4.8]{chevyrev_characteristic_2016} and~\cite{boedihardjo_signature_2016}, ~\Cref{lem:tleq_parametrized_boundary} shows that such a 1-connection must exist.
\end{proof}

\subsection{Characterizing the Interior}
We consider $\cmgl^{0,m,p}$ to capture information purely about the interior of a surface. From the fake-flatness conditions for $\con = (\cona, \conb, \conc) \in \cM^{0,m,p}(\R^{d+2})$, we must have $[\cona,\cona] = 0$ and $[\conb, \conb]= 0$, and no conditions on $\conc$. Because $\cona$ and $\conb$ are flat connections on the contractible space $\R^{d+2}$, the holonomy over loops is trivial; thus $\cmb H^{\con}(\obX) = I$.
Furthermore, the map $\phi$ of the 2-vector space is trivial, so the matrix surface development \Cref{eq:msh_decomp1} reduces to
\[
    \partial_t H^\con_{s,t}(\obX) = \partial_t Z^\con_{s,t}(\obX).
\]
Thus $H^\con(\obX) = Z^\con(\obX)$, so it suffices to consider the computation of $Z^{\con}(\obX)$.

We begin with the parametrized area process $A(\obX)$ of a surface $\bX \in C^{\rho}_0([0,1]^2, \V)$.
Let $A^{s,i}(\obX)$ and $A^{t,i}(\obX)$ denote the components of the area process $A(\obX)$ which correspond to the areas of $(s, \bX^i)$ and $(t, \bX^i)$ respectively. 
These area components can be computed as
\[
    A^{s,i}_{s,t}(\obX) = \int_0^s (\bX^i_{s', t} - \bX^i_{s',0}) ds' \quad \text{and} \quad A^{t,i}_{s,t}(\obX) = -\int_0^t (\bX^i_{s, t'} - \bX^i_{0,t'}) dt'.
\]
We will only need to consider one of the two area processes, so we focus on $A^{s,i}_{s,t}(\obX)$.
If $\bX, \bY \in C^{\rho}_0([0,1]^2, \V)$ are distinct surfaces which have the same boundary, $\partial \bX = \partial \bY$, then there must exist some $i$ such that $A^{s,i}(\obX) - A^{s,i}(\obY) \neq 0$. Note that these components of the area process are \emph{linear} with respect to the surface,
\[
    A^{s,i}_{s,t}(\obX) - A^{s,i}_{s,t}(\obY) = A^{s,i}_{s,t}(\overline{\bX - \bY}).
\]
and thus it suffices to consider the difference $\bX - \bY$ which has trivial boundary. Let $C^\rho_{\square}([0,1]^2, \V) \label{pg:triv_boundary}$ denote the space of surfaces with trivial boundary. 

\begin{lemma} \label{lem:parametrized_area_nontrivial}
    Let $\bX \in C^{\rho}_{\zbdy}([0,1]^2, \V)$. If $\bX \neq 0$, then there exist some $i \in [d]$, $0 < u, u' < 1$ and $0 < v, v' < 1$ such that
    \[
        \square_{u,u';v,v'}[A^{s,i}(\obX)] \neq 0.
    \]
\end{lemma}
\begin{proof}
    Because $\bX$ is nontrivial, there exists some $u', v' \in [0,1]$ and $i \in [d]$ such that
    \[
        A^{s,i}_{u',v'}(\obX) = \int_0^{u'} \bX^i_{s, v'} ds' \neq 0.
    \]
    In particular, this implies $\square_{0,u';0,v'}[A^{s,i}(\obX)] = A^{s,i}_{u',v'}(\obX) \neq 0$. Finally, since $\square_{u,u';v,v'}[A^{s,i}(\obX)]$ is continuous in $u$ and $v$, the result holds.
\end{proof}

This lemma implies that the linear map
\[
    \bX \mapsto (A^{s,1}(\obX), \ldots, A^{s,d}(\obX))
\]
is injective. Now, suppose that $\bX, \bY \in C^{\rho}_0([0,1]^2, \V)$ such that $\bdy \bX = \bdy \bY$, but $\bX \neq \bY$, and let $i \in [d]$ be the coordinate such that $A^{s,i}(\obX) \neq A^{s,i}(\obY)$. We will consider a matrix 2-connection $\con = (\cona, \conb, \conc) \in \cM^{0,m,p}$ such that the only nontrivial component of $\conc$ is $\conc^{s,i}$, so that $Z^\con(\obX)$ is
\begin{align} \label{eq:Z_restricted_2_connection}
    Z^\con_{s,t}(\obX) = \int_0^t \int_0^s F^{\cona}(\obx^{s',t'}) \conc^{s,i} F^{\conb}(\obx^{s',t'})^{-1} dA^{s,i}_{s',t'}(\obX).
\end{align}
In~\Cref{prop:2d_integral_characterization}, we show that if $A^{s,i}_{s,t}(\obX) \neq 0$ there exist some $a, b \in \N$ such that
\[
    \int_0^1 \int_0^1 s^a t^b dA^{s,i}_{s,t}(\obX) \neq 0.
\]
Thus, it remains to show that for any $a,b \in \N$, there exists a choice of 2-connection such that the integrand of~\Cref{eq:Z_restricted_2_connection} such that the integrand is $s^a t^b$. 

\begin{proposition} \label{prop:polynomial_integrals}
    For any $a, b \in \N$ and $i \in [d]$, there exists $\con = (\cona, \conb, \conc) \in \cM^{0,a+1,b+1}(\pV)$ with $\conc = \conc^{s,i} ds \wedge dx_i$ such that for any $\bX\in C^{\rho}_0([0,1]^2, \V)$, there exists an entry of
    \[
        T^\con(\bX) = F^{\cona}(\obx^{s,t})\cdot \conc^{s,i}\cdot F^{\conb}(\obx^{s,t})^{-1} \in \cmgl^{0,a+1,b+1}_2
    \]
    which is a polynomial with leading term $s^a t^b$.
\end{proposition}
\begin{proof}
    We choose $\cona$ and $\conb$ such that the only nonzero components are $\cona^s$ and $\conb^t$ respectively. Note that this satisfies the fake-flatness condition. Then, by definition of the tail path $\obx^{s,t}$, we have
    \[
        F^{\cona}(\obx^{s,t}) = \exp(\cona^s s) \quad  \andd \quad  F^{\conb}(\obx^{s,t}) = \exp(\conb^t t). 
    \]
    The entries of $F^{\cona}(\obx^{s,t})$ will be polynomials in $s$ if $\cona^s$ is chosen to be a strictly upper triangular matrix. Define the matrices $U_{m, k} \in \Mat_{m,m}$ which is $1$ above the $k^{th}$ upper diagonal, by
    \[
        [U_{m,k}]_{i,j} \coloneqq \left\{ \begin{array}{cl}
            1 & : j  \ge i+k \\
            0 &: \text{otherwise}.
        \end{array}\right.
    \]
    For example,
    \[
        U_{3,1} \coloneqq \pmat{0 & 1 & 1 \\ 0 & 0 & 1 \\ 0 & 0 & 0} \quad \andd \quad U_{3,2} \coloneqq \pmat{0 & 0 & 1 \\ 0 & 0 & 0 \\ 0 & 0 & 0}.
    \]
    Note that we have $U_{m,1}^k = U_{m,k}$ and $U_{m,k} = 0$ for $k \ge m$.
    Now, we set $\cona^s = U_{a+1,1}$, and we obtain
    \[
        F^{\cona}(\obx^{s,t}) = \exp(U_{a+1,1} s) = \sum_{k=0}^{a} \frac{U_{a+1,k} s^k}{k!}.
    \]
    We express this entry-wise as
    \[
        [\exp(U_{a+1,1} s)]_{i,j}\coloneqq \left\{ \begin{array}{cl}
            E_{j-i}(s) & : j  \ge i \\
            0 &: \text{otherwise}
        \end{array}\right.
        \quad \text{where} \quad
        E_k(s) \coloneqq \sum_{j=0}^k \frac{s^j}{j!}
    \]
    is the truncated Taylor series of the exponential.
    Thus, the $(1, a+1)$ entry of $F^{\cona}(\obx^{s,t})$ is $E_a(s)$ which has leading term $\frac{s^a}{a!}$. Similarly, we will choose $\conb^t = -U_{b+1,1}$ so that the $(1, b+1)$ entry of $F^{\conb}(\obx^{s,t})^{-1}$ is $E_b(t)$, which has leading term $\frac{t^b}{b!}$. 
    Finally, we choose $\conc^{s,i} \in \Mat_{a+1,b+1}$ to be the matrix with a $c a! b!$ in the $(a+1,1)$ entry, and $0$ elsewhere. By direct computation, we find that the $(1,b+1)$ entry has the desired polynomial,
    \[
        [F^{\cona}(\obx^{s,t})\cdot \conc^{s,i}\cdot F^{\conb}(\obx^{s,t})^{-1}]_{1,b+1} = ca! b! E_a(s) E_b(t) = c s^a t^b + \text{ lower degree terms}.
    \]
\end{proof}

\begin{proposition} \label{prop:injectivity_interior}
    Let $\bX, \bY \in C^{\rho}_{\zor}([0,1]^2, \V)$ such that $\bdy \bX = \bdy \bY$ and $\bX \neq \bY$. There exists a matrix 2-connection $\con = (\cona, \conb, \conc) \in \cM^{0,m,p}(\pV)$ such that
    \[
        H^{\con}(\obX) \neq H^{\con}(\obY).
    \]
\end{proposition}
\begin{proof}
    By~\Cref{lem:parametrized_area_nontrivial}, there exists some $i \in [d]$ such that $A^{s,i}(\overline{\bX - \bY}) : [0,1]^2 \to \R$ is nontrivial. Then by~\Cref{prop:2d_integral_characterization}, there exists some $a, b \in \N$ such that
    \begin{align} \label{eq:inj_integral_assumption}
        \int_0^1 \int_0^1 s^a t^b dA^{s,i}_{s,t}(\overline{\bX - \bY}) \neq 0.
    \end{align}
    We choose the lowest degree $s^a t^b$ such that this is true.
    Next, by~\Cref{prop:polynomial_integrals}, there exists a matrix 2-connection $\con \in \cM^{0, a+1, b+1}(\pV)$ such that
    \[
        [F^{\cona}(\obx^{s,t}) \cdot\conc^{s,i} \cdot F^{\conb}(\obx^{s,t})^{-1}]_{1,b+1} = c s^a t^b + \text{ lower degree terms}.
    \]
    This integrand is independent of $\bX$, since it uses only the parametrization. Thus, we have
    \[
        [Z^\con(\obX) - Z^\con(\obY)]_{1,b+1} = \int_0^1 \int_0^1 s^a t^b dA^{s,i}_{s,t}(\overline{\bX - \bY}) \neq 0,
    \]
    where the lower degree terms integrate to $0$ by the assumption that $a,b \in \N$ are chosen where $s^a t^b$ is the lowest degree term such that~\Cref{eq:inj_integral_assumption} holds. Then, since $H^\con(\obX) = Z^\con(\obX)$ for $\con \in \cM^{0,m,p}$, we obtain the desired result.
\end{proof}

\section{Characteristicness and Random Surfaces} \label{sec:characteristicness}

In this section, we construct the desired characteristic functions for surface-valued probability measures.
We will show that expectations of surface development can be used to characterize the law of such measures using the duality between universality and characteristicness.

\begin{definition}
    Suppose $\cX$ is a topological space, and let $\cF \subset \F^\cX$ be a topological vector space where $\F = \R$ or $\C$, such that the topological dual $\cF'$ is a space of measures. Let $\cF_0 \subset \cF$ be another topological vector space. 
    \begin{enumerate}
        \item The set $\cF_0$ is \emph{universal} with respect to $\cF$ if $\cF_0$ is dense in $\cF$.
        \item The set $\cF_0$ is \emph{characteristic} with respect to $\cP \subset \cF'$ if the map $\cE: \cP \to \cF_0'$, defined by
        \[
            \cE(\mu)(f) = \int f(x) d\mu(x)
        \]
        is injective. 
    \end{enumerate}
\end{definition}

The duality of these two properties have been studied from the language of feature maps~\cite[Theorem 7]{chevyrev_signature_2022} and kernels~\cite[Theorem 6]{simon-gabriel_kernel_2018} in the machine learning literature. The proof of the following theorem follows exactly from~\cite[Theorem 6]{simon-gabriel_kernel_2018}.

\begin{theorem}{\cite[Theorem 6]{simon-gabriel_kernel_2018}} \label{thm:duality}
    Suppose $\cF$ is a locally convex topological vector space. Then $\cF_0$ is universal to $\cF$ if and only if $\cF_0$ is characteristic to $\cF'$. 
\end{theorem}

As discussed for the case of paths in~\cite{chevyrev_signature_2022,cuchiero_global_2023}, one of the difficulties in proving universality is the fact that $C^{\rho}_{\zor}([0,1]^2, \V)$ is not locally compact, and thus standard Stone-Weierstrass arguments do not apply. Following~\cite{chevyrev_signature_2022}, we will rectify this issue by considering strict topologies. \medskip

For a metric space $\cX$, let $C_b(\cX, \C) \label{pg:Cb}$ denote the space of continuous bounded functions on $\cX$ valued in $\C$. Further, let $B_0(\cX)$ denote the space of functions $\theta : \cX \to \R$ which \emph{vanish at infinity}; ie.~for all $\epsilon > 0$, there exists a compact $K_\epsilon \subset \cX$ such that $\sup_{x \in \cX \backslash K_\epsilon} |\theta(x)| < \epsilon$. 

\begin{definition} \label{def:strict_topology}
    The \emph{strict topology} on $C_b(\cX, \C)$ is the topology generated by the seminorms
    \[
        p_\theta(f) = \sup_{x \in \cX} |f(x) \theta(x)|
    \]
    for all $\theta \in B_0(\cX)$. 
\end{definition}

\begin{theorem}\label{thm:strict_topology}
    Let $\cX$ be a metric space.
    \begin{enumerate}
        \item \label{itm:strict1} Let $\cF_0$ be a subalgebra of $C_b(\cX, \C)$ such that for all $x \neq y \in \cX$, there exists $f \in \cF_0$ such that $f(x) \neq f(y)$, and for all $x \in \cX$ there exists $f \in \cF_0$ such that $f(x) \neq 0$, then $\cF_0$ is dense in $C_b(\cX, \C)$ under the strict topology.
        \item \label{itm:strict2} The topological dual of $C_b(\cX, \C)$ equipped with the strict topology is the space of complex Radon measures on $\cX$.
    \end{enumerate}
\end{theorem}
\begin{proof}
    The real version of this theorem is proved in~\cite{giles_generalization_1971}. \Cref{itm:strict1} is obtained from~\cite[Theorem 4.3]{summers_general_1971}. For~\Cref{itm:strict2}, the topological dual of $C_b(\cX, \C)$ is shown to be the space of tight linear functionals in~\cite{rooij_tight_1967}, which are equivalent to Radon measures for completely regular spaces (see the proof of~\cite[Theorem 7.10.6]{bogachev_measure_2006}). 
\end{proof}

\subsection{Universality and Characteristicness for Unparametrized Surfaces}
We begin by considering characteristic functions for unparametrized (thin homotopy classes of) surfaces in $\thinsurface^{0,\rho}(\V)$.
Recall from~\Cref{def:matrix_2connection} that $\cM^{n,m,p}(\V)$ denotes the space of matrix 2-connections on $\V$ valued in $\cmgl^{n,m,p}$. In order to reduce the number of parameters, we will only consider connections in $\cM^n(\pV) \coloneqq \cM^{n,n,n}(\pV)$, since $\cmgl^{n,m,p}$ embeds in $\cmgl^N \coloneqq \cmgl^{N,N,N}$ for all $n,m,p \leq N \label{pg:A_not_res}$ (see~\Cref{apxsec:direct_sums}).
We define the \emph{space of surface development functions} by
\begin{equation}
    \cA \coloneqq \bigg\{ \sum_{k=1}^r c_k \exp\left( \im \langle \ell_k, H^{\con_k}(\overline{\,\cdot\,}) \rangle \right)  \, : \, c_k \in \C,\, \con_k \in \cM^{n_k}(\V), \,\ell_k \in \Mat_{2n_k, 2n_k}\bigg\}
\end{equation}
where 
\[
    \bX \mapsto \sum_{k=1}^r c_k \exp\left( \im \langle \ell_k, H^{\con_k}(\obX) \rangle \right) : \thinsurface^{0,\rho}(\V) \to \C. 
\]
Note in particular that $\cA \subset C_b(\thinsurface^{0,\rho}(\V), \C)$. 
\begin{remark} \label{rem:failure_alg_structure}
    In the 1D setting of path development, one can directly consider products of linear functionals by using tensor products of the path holonomies,
    \[
        \langle \ell_1, F^{\cona_1}(\bx)\rangle \cdot \langle \ell_2, F^{\cona_2}(\bx) \rangle = \langle \ell_1\otimes \ell_2, F^{\cona_1 \otimes I + I \otimes \cona_2}(\bx)\rangle. 
    \]
    However, due to the structure of tensor products of 2-vector spaces, we cannot generalize this property (see~\Cref{apxsec:2_tensor_product} for further details) which is why we use the more complex function class $\cA$.
\end{remark}

In order to deal with the separation property, we will define a new equivalence class,
\begin{align}
    \wthinsurface^{0,\rho}(\V) \coloneqq \thinsurface^{0,\rho}(\V)/\sim_{\text{he}} \quad \text{where} \quad \bX \sim_{\text{he}} \bY \quad \text{if} \quad f(\bX) = f(\bY) \text{ for all } f \in \cA.
\end{align}

\begin{theorem} \label{thm:univ_char_thinhom}
    The function space $\cA$ defined above has the following properties.
    \begin{enumerate}
        \item \textbf{(Universality)} The space $\cA$ is dense in $C_b(\wthinsurface^{0,\rho}(\V), \C)$ with the strict topology.
        \item \textbf{(Characteristicness)} Given two probability measures $\mu, \nu \in \cP(\wthinsurface^{0,\rho}(\V))$ such that $\mu \neq \nu$, there exists some $\con \in \cM^n(\V)$ and $\ell \in \Mat_{2n,2n}$ such that
        \begin{align} \label{eq:measures_not_equal}
            \E_{\bX \sim \mu}\Big[ \exp\left( \im \langle \ell, H^{\con}(\bX) \rangle \right)\Big] \neq \E_{\bY \sim \nu}\Big[ \exp\left( \im \langle \ell, H^{\con}(\bY) \rangle \right)\Big].
        \end{align}
    \end{enumerate}
\end{theorem}
\begin{proof}
    First, we note that $\cA$ is a subalgebra of $C_b(\wthinsurface^{0,\rho}(\V), \C)$. Indeed, let $\con_1 \in \cM^{n_1}(\V)$, $\con_2 \in \cM^{n_2}(\V)$, and $\ell_1 \in \Mat_{2n_1,2n_1}$, $\ell_2 \in \Mat_{2n_2,2n_2}$. Then using~\Cref{lem:direct_sum_sh}, we have
    \begin{align*}
        \exp\left( \im \langle \ell_1, H^{\con_1}(\bX) \rangle \right) \cdot \exp\left(\im \langle \ell_2, H^{\con_2}(\bX) \rangle \right) = \exp \left( \im \left\langle \pmat{\ell_1 & 0 \\ 0 & \ell_2}, H^{\con_1 \oplus \con_2}(\bX) \right\rangle \right),
    \end{align*}
    where $\con_1 \oplus \con_2$ is defined in~\Cref{apxsec:direct_sums}.
    Furthermore, $\cA$ contains the constant functions and separates points by definition of $\wthinsurface^{0,\rho}(\V)$. 
    Thus, by the strict Stone-Weierstrass theorem (\Cref{thm:strict_topology}), the first statement holds. Then, by duality between universality and characteristicness from~\Cref{thm:duality}, we know that $\cA$ is characteristic with respect to the dual of $C_b(\wthinsurface^{0,\rho}(\V), \C)$, which is the space of complex Radon measures (\Cref{thm:strict_topology}). Because this contains the space of probability measures, there exists some $f \in \cA$ such that $\E_{\bX \sim \mu}[f(\bX)] \neq \E_{\bY \sim \nu}[f(\bY)]$. 
\end{proof}

While the fact that path development separates thin homotopy classes is well-known in the case of paths~\cite{chevyrev_characteristic_2016,boedihardjo_signature_2016}, it has only recently been proven in the setting of piecewise linear surfaces~\cite{bischoff_thin_2025}. Let $\thinsurface^{\text{PL}}(\V)$ denote the thin homotopy classes of piecewise linear surfaces. Then, the characterization of thin homotopy classes in~\cite[Theorem 6.26]{bischoff_thin_2025} together with the universal property of the surface signature in~\cite[Theorem 4.29]{lee_surface_2024}, we obtain the following corollary.

\begin{corollary}
    Given two probability measures $\mu, \nu \in \cP(\thinsurface^{\text{PL}}(\V))$ such that $\mu \neq \nu$, there exists some $\con \in \cM^n(\V)$ and $\ell \in \Mat_{2n,2n}$ which separates the measures~\Cref{eq:measures_not_equal}.
\end{corollary}

\subsection{Universality and Characteristicness for Parametrized Surfaces}

Next, we will discuss characteristic functions for parametrized surfaces, where we now consider connections on $\pV$. Note that the space of 2-connections $\cM^n(\pV)$ is not a linear space, due to the quadratic fake-flatness conditions from~\Cref{def:matrix_2connection},
\begin{align} \label{eq:quadratic_semiflat_conditions}
    [C^i, C^j] = 0 \quad \andd \quad [E^i, E^j] = 0
\end{align}
for the $C$ and $E$ blocks of the 2-connection. In order to further reduce the number of connections we must consider, we define a restricted \emph{linear} space of 2-connections. Using the block notation of~\Cref{def:matrix_2connection}, we define
\begin{align}
    \cM_{\res}^n(\pV) \coloneqq \Big\{(\cona, \conb, \conc) \in \cM^n(\pV) \, : \, C^i = 0,\,  E^j =0 \text{ for }  i \in [d] \cup\{t\}, \, j \in [d] \cup \{ s\} \Big\}.
\end{align}
In particular, these are 2-connections where the only nonzero $C \label{pg:cM_res}$ block of $\cona$ is $C^s$, and the only nonzero $E$ block of $\conb$ is $E^t$; thus the quadratic conditions of~\Cref{eq:quadratic_semiflat_conditions} are  satisfied. Furthermore, $\cM_{\res}^n(\pV)$ is a linear space, which can be written as
\begin{align} \label{eq:restricted_matrix_2_connections}
    \cM_{\res}^n(\pV) \cong \underbrace{\Mat_{n,n}^{d+2}}_{A^i} \oplus \underbrace{\Mat_{n,n}^{d+2}}_{B^i} \oplus \underbrace{\Mat_{n,n}}_{C^s} \oplus \underbrace{\Mat_{n,n}^{d+2}}_{D^i} \oplus \underbrace{\Mat_{n,n}}_{E^t} \oplus \underbrace{\Mat_{n,n}^{\binom{d+2}{2}}}_{U^{i,j}}.
\end{align}
Note that the 2-connections required for~\Cref{thm:sh_injectivity} are contained in $\cM^n_{\res}(\pV) \label{pg:A_res}$. We define the restricted space of surface development functions by
\begin{equation}
    \cA_{\res} \coloneqq \bigg\{ \sum_{k=1}^r c_k \exp\left( \im \langle \ell_k, H^{\con_k}(\overline{\,\cdot\,}) \rangle \right)  \, : \, c_k \in \C,\, \con_k \in \cM^{n_k}_{\res}(\pV), \,\ell_k \in \Mat_{2n_k, 2n_k}\bigg\}
\end{equation}
where 
\[
    \bX \mapsto \sum_{k=1}^r c_k \exp\left( \im \langle \ell_k, H^{\con_k}(\obX) \rangle \right) : C^{\rho}_{\zor}([0,1]^2, \V) \to \C. 
\]

\begin{theorem} \label{thm:univ_char}
    The function space $\cA_{\res}$ defined above has the following properties.
    \begin{enumerate}
        \item \textbf{(Universality)} The space $\cA_{\res}$ is dense in $C_b(C^{\rho}_0([0,1]^2, \V), \C)$ with the strict topology.
        \item \textbf{(Characteristicness)} Given two probability measures $\mu, \nu \in \cP(C^{\rho}_0([0,1]^2, \V))$ such that $\mu \neq \nu$, there exists some $\con \in \cM^n_{\res}(\pV)$ and $\ell \in \Mat_{2n,2n}$ such that
        \[
            \E_{\bX \sim \mu}\Big[ \exp\left( \im \langle \ell, H^{\con}(\overline{\bX}) \rangle \right)\Big] \neq \E_{\bY \sim \nu}\Big[ \exp\left( \im \langle \ell, H^{\con}(\overline{\bY}) \rangle \right)\Big].
        \]
    \end{enumerate}
\end{theorem}
\begin{proof}
    This proof is largely the same as~\Cref{thm:univ_char_thinhom}.
    First, we note that $\cA_{\res}$ is a subalgebra of $C_b(C^{\rho}_0([0,1]^2, \V), \C)$ by~\Cref{lem:direct_sum_sh}. 
    Furthermore, $\cA_{\res}$ contains the constant functions and this separates points by~\Cref{thm:sh_injectivity}; in particular the 2-connections from~\Cref{prop:injectivity_boundary} and~\Cref{prop:injectivity_interior} are contained in $\cM_\res$. 
    The first statement holds by~\Cref{thm:strict_topology} and the second statement holds by~\Cref{thm:duality}.
\end{proof}

\begin{remark}
    Another possible approach to obtain an algebra structure for surface development is to consider the polynomial functions generated $\langle \ell, H^\con(\obX)\rangle$ for all $\con \in \cM^n_{\res}(\R^{d+2})$ and $\ell \in \cmgl_1^n$, with $n \in \N$. Then, one can consider weighted topologies on function spaces, and follow the methodology in~\cite{cuchiero_global_2023} to obtain characteristic functions, though this approach will restrict to laws of random surfaces which satisfy certain moment conditions. 
\end{remark}

As an immediate application, we can characterize the law of fractional Brownian sheets.

\begin{definition} \label{def:fbs}
    The \emph{fractional Brownian sheet} $B : [0,1]^2 \to \R$ parametrized by Hurst parameter $h \in (0,1)$ is a Gaussian field on $[0,1]^2$ with $B_{s,0} = B_{0,t} = 0$ and with covariance function
    \[
        Q^h(s_1, t_1, s_2, t_2) \coloneqq \E \Big[ B_{s_1, t_1} \cdot  B_{s_2, t_2}\Big] = \left( s_1^{2h} + t_1^{2h} - |s_1 - t_1|^{2h} \right) \cdot \left( s_2^{2h} + t_2^{2h} - |s_2 - t_2|^{2h} \right).
    \]
\end{definition}

\begin{corollary}
    Let $B: [0,1]^2 \to \V$ be a fractional Brownian sheet with Hurst parameter $h \in (\frac12, 1)$ with independent components. Then, the quantity
    \[
    \E_{\bX \sim B}\Big[ \exp\left( \im \langle \ell, H^{\con}(\overline{\bX}) \rangle \right)\Big]
    \]
    for all $\bomega \in \cM^n_{\res}(\pV)$, $\ell \in \Mat_{2n,2n}$ and $n \in \N$ characterize the law of $B$ within all probability measures valued in $C^\rho_0([0,1]^2, \V)$ for $\rho \in (\frac12, h)$.
\end{corollary}
\begin{proof}
    By~\cite[Example 2]{harang_extension_2021}, $B \in C^\rho([0,1]^2, \V)$ almost surely for any $\rho \in (\frac12, h)$. Then, the result follows from~\Cref{thm:univ_char}.
\end{proof}

\subsection{Metrics for Random Surfaces} \label{ssec:metrics_random}

We will now use the characteristic property to obtain a metric for random surface, which requires us to summarize the information from the surface development over all restricted matrix 2-connections. Using a method similar to~\cite{lou_pcf-gan_2023} in the case of paths, we will achieve this by considering the expected distance between the surface development functions with respect to a probability measure over the space of connections and the space of linear functionals. We denote the space of probability measures on $\rho$-H\"older surfaces by $\cP^\rho \coloneqq \cP ( C^{\rho}_0([0,1]^2, \V))$ to simplify notation 

\begin{definition} \label{def:sh_metrics}
    For $n \in \N$, define $\td_n :C^{\rho}_0([0,1]^2, \V) \times C^{\rho}_0([0,1]^2, \V) \times \cM_{\res}^n(\pV) \times \Mat_{2n,2n} \to \R$ by
    \[
        \td_n(\bX, \bY; \con, \ell) \coloneqq \left| \exp\left( \im \langle \ell, H^{\con}(\obX) \rangle \right) - \exp\left( \im \langle \ell, H^{\con}(\obY) \rangle \right)\right|.
    \]
    Let $\Xi^n$ be the standard Gaussian measure on $\cM^n_{\res}$ and $\Theta^n$ be the uniform measure on the unit ball of $\Mat_{n,n}$. The \emph{level $n$ surface development pseudometric}, $d_n : \cP^\rho \times \cP^\rho \to \R$, is defined by
    \begin{align}
        d_n(\mu,\nu) \coloneqq \E_{\con \sim \Xi^n}\E_{\ell \sim \Theta^n} \E_{\bX \sim \mu} \E_{\bY \sim \nu}\Big[ \td_n(\bX,\bY; \con, \ell)\Big].
    \end{align}
    Finally, the \emph{surface development metric}, denoted $d: \cP^\rho \times \cP^\rho \to \R$ is defined by
    \begin{align}
        d(\mu, \nu) \coloneqq \sum_{n=1}^\infty \frac{d_n(\mu, \nu)}{n!}.
    \end{align}
\end{definition}

\begin{proposition} \label{prop:sh_metric}
    The functions $d_{n}$ for $n \geq 1$ in~\Cref{def:sh_metrics} are pseudometrics and $d$ is a metric.
\end{proposition}
\begin{proof}
    Symmetry is straightforward. The function $d_{n}$ satisfies the triangle inequality because $\td_n(\cdot, \cdot; \con, \ell)$ satisfies it and the expectation is linear. Thus, $d_{n}$ is a pseudo-metric.

    Next, $d$ is a pseudometric since each $d_n$ is a pseudometric. Suppose that $d(\mu, \nu) = 0$. This implies that $d_{n}(\mu, \nu) = 0$ for all $n \in \N$. Then, since each $\Xi^{n}$ has full support, this implies that 
    \[
        \E_{\bX \sim \mu}\Big[ \exp\left( \im \langle \ell, H^{\con}(\overline{\bX}) \rangle \right)\Big] = \E_{\bY \sim \nu}\Big[ \exp\left( \im \langle \ell, H^{\con}(\overline{\bY}) \rangle \right)\Big]
    \]
    for all $\con \in \cM^{n}_{\res}(\pV)$, $\ell \in \Mat_{2n,2n}$ and $n \in \N$. By~\Cref{thm:univ_char}, $\mu = \nu$, so $d$ is a metric.
\end{proof}

\begin{theorem} \label{thm:sh_metric_weak_convergence}
    Weak convergence in $\cP^\rho$ implies convergence in $d$. Furthermore, given $\cK \subset \cP^\rho$ which is compact in the weak topology, convergence in $d$ implies weak convergence.
\end{theorem}
\begin{proof}
    Suppose $\mu \in \cP^\rho$, and $\mu_k \in \cP^\rho$ is a measure which weakly converges to $\mu$. We note that by~\Cref{thm:main_sd_holder}, the function $\td_n$ from~\Cref{def:sh_metrics} is continuous, where $\cM^n_{\res}(\pV)$ and $\Mat_{2n,2n}$ are equipped with the usual Euclidean topology. Then, for each $n \in \N$, we have
    \begin{align*}
        \lim_{k \to \infty} d_n(\mu, \mu_k) & =  \lim_{k \to \infty} \E_{\con \sim \Xi^n}\E_{\ell \sim \Theta^n} \E_{\bX \sim \mu} \E_{\bY \sim \mu_k}\Big[ \td_n(\bX,\bY; \con, \ell)\Big] \\
        & = \E_{\con \sim \Xi^n}\E_{\ell \sim \Theta^n} \Big[ \lim_{k \to \infty} \E_{\bX \sim \mu}\E_{\bY \sim \mu_k}\Big[ \td_n(\bX,\bY; \con, \ell)\Big] \Big] \\
        & = 0,
    \end{align*}
    where the second line is given by the dominated convergence theorem, and the last equality is by the weak convergence of $\mu_k$ to $\mu$. Thus, $\lim_{k \to \infty} d(\mu, \mu_k) = 0$. Because the weak topology and the topology induced by $d$ are comparable on a compact $\cK \subset \cP^\rho$, they must coincide on $\cK$. 
\end{proof}

\appendix
\clearpage
\section{Notation} \label{apx:notation}

{\small
\begin{longtable}
    {ccc}
     \toprule
    Symbol & Description & Page\\ \midrule
    \multicolumn{3}{c}{Function Spaces} \\ \midrule
    $C_b(\cX, \mathbb{K})$ & continuous bounded functions & \pageref{pg:Cb} \\
    $C^\infty([0,1]^k, \V)$ & space of smooth paths ($k=1$)/surfaces ($k=2$) with sitting instants  & \pageref{pg:smooth_paths}, \pageref{pg:smooth_surfaces}\\
    $C^\rho([0,1]^k \V)$ & space of $\rho$-H\"older paths/surfaces&  \pageref{pg:holder_space} \\
    $C^{0, \rho}([0,1]^k, \V)$ & space of $\rho$-H\"older paths/surfaces in  smooth closure & \pageref{pg:smoothcl_holder}\\
    $\mathsf{T}_k(\V)$ & thin homotopy classes of smooth paths/surfaces & \pageref{eq:thin_paths}, \pageref{eq:thin_surfaces}\\
    $\mathsf{T}^{0,\rho}_k(V)$ & thin homotopy classes of $\rho$-H\"older paths/surfaces in smooth closure & \pageref{def:thin_smoothcl_holder}\\ 
    $C^{R}_0([0,1]^k, \V)$ & space of smooth paths/surfaces in $C^R([0,1]^k, \V)$ based at origin & \pageref{pg:based}\\
    $C^{R}_\square([0,1]^k, \V)$ & space of smooth paths/surfaces in $C^R([0,1]^k, \V)$ with trivial boundary & \pageref{pg:triv_boundary} \\\midrule
    \multicolumn{3}{c}{Double Groups and Crossed Modules} \\ \midrule
    $\thindg(\V)$ & thin double group of smooth surfaces& \pageref{def:thin_double_group}\\
    $\thindg^{0,\rho}(\V)$ & thin double group of $\rho$-H\"older surfaces in smooth closure & \pageref{prop:thin_holder_dg}\\
    $\dgcm(\cmG)$ & double group associated to crossed module $\cmG$ & \pageref{prop:dgcm} \\
    $\cmGL^{n,m,p}$ & general linear crossed module for $\R^{n+m} \to \R^{n+p}$ & \pageref{eq:general_linear_crossed_module} \\
    $\cmgl^{n,m,p}$ & general linear crossed module of Lie algebras for $\R^{n+m} \to \R^{n+p}$ & \pageref{pg:cmgl} \\
    $\delta, \gt$ & crossed module boundary map and action (for groups and Lie algebras) & \pageref{def:GCM}, \pageref{def:DCM} \\
    $\concat_h, \concat_v$ & horizontal and vertical concatenation of surfaces & \pageref{eq:surface_concatenation} \\
    $\odot_h, \odot_v$ & horizontal and vertical composition in other double groups & \pageref{prop:dgcm} \\ \midrule
    \multicolumn{3}{c}{Path and Surface Development} \\ \midrule
    $\bomega$ & $\bomega = (\alpha, \gamma)$ fake flat translation-invariant 2-connection & \pageref{def:2_connection}\\
    $\cM^{n,m,p}$ & space of 2-connections valued in $\cmgl^{n,m,p}$ & \pageref{def:matrix_2connection}\\
    $\cM^{n}_\res$ & restricted linear space of 2-connections valued in $\cmgl^{n,n,n}$ & \pageref{pg:cM_res}\\
    $F^\alpha$ & path development map & \pageref{pg:path_development}\\
    $\bF^\bomega, H^\bomega$ & surface development functor and map & \pageref{eq:sd_functor} \\
    $\cA, \cA_{\res}$ & (restricted) space of surface development functions & \pageref{pg:A_not_res}, \pageref{pg:A_res}\\\midrule
    \multicolumn{3}{c}{Norms and Seminorms} \\ \midrule
    $\|\bx\|_\rho, \llb\bx\rrb_\rho$ & $\rho$-H\"older norm and seminorm for paths & \pageref{pg:holder_space}\\
    $\|\bX\|_\rho, \llb\bX\rrb_\rho, |\bX|_{\rho, (i)}$ & $\rho$-H\"older norm and seminorms for surfaces & \pageref{pg:holder_space}\\
    $\|\bomega\|_\F$ & Frobenius norm & \pageref{eq:frob_norm_matrix_2_connection} \\
    $\|V\|_{\Lip}$ & Lipschitz norm of vector fields & \pageref{pg:lip_norm}\\ \midrule
\end{longtable}
}

\section{Further Details on Thin Double Group} \label{apx:thin_dg}

\begin{example} \label{ex:2d_thin_homotopy}
Here, we provide some examples of thin homotopy invariance of $\bX: [0,1]^2 \to \R^d$.
\begin{enumerate}

    \item \textbf{Reparametrizations of the domain.} Suppose $\psi: [0,1]^2 \to [0,1]^2$ is a smooth bijection which preserves the corners and boundary, $\psi(\partial [0,1]^2) = \partial [0,1]^2$. Then $\bX \circ \psi$ is thin homotopy equivalent to $\bX$. Indeed, we can first define a homotopy $\Psi: [0,1]^3 \to [0,1]^2$ between $\psi$ and the identity $\id: [0,1]^2 \to [0,1]^2$ by
    \[
        \Psi_{s,t,u} \coloneqq u \psi_{s,t} + (1-u) \id_{s,t}.
     \]
     Then, $\bighomotopy: [0,1]^3 \to \R^d$ defined by $\bighomotopy_{s,t,u} = \bX \circ \Psi_{s,t,u}$ is a thin homotopy from $\bX$ to $\bX \circ \psi$. 

    \item \textbf{Folding.} This is a type of invariance which arises when considering horizontal and vertical inverses as described above. In particular, for any smooth surface $\bX$, we have $\bX \concat_h \bX^{-h} \sim_{\thinhom} 1^h_{\bdy_l \bX}$, and similarly for vertical composition. 
\end{enumerate}
\end{example}

Now, we will provide some further details on the double group operations in the thin double group (\Cref{def:thin_double_group}).\medskip

\textbf{Boundary Maps.}
By definition, thin homotopy equivalent surfaces have thin homotopy equivalent boundaries, so the boundaries from~\Cref{eq:strict_boundaries} are well-defined.\medskip

\textbf{Identity Squares.}
Given a thin homotopy equivalence class of paths $[\bx] \in \thinpath(\V)$ with a representative $\bx: [0,1] \to \R^d$, the horizontal and vertical identity squares are exactly $[1^h_\bx], [1^v_\bx] \in \thinsurface(\V)$. To verify this is well defined, suppose $\bx \sim_{\thinhom} \by$, and let $\homotopy: [0,1]^2 \to \R^d$ be a thin homotopy between $\bx$ and $\by$. Then, we note that $\bighomotopy:[0,1]^3 \to \R^d$ defined by $\bighomotopy_{u,s,t} = \homotopy_{u,t}$ is a thin homotopy between $1^h_\bx$ and $1^h_\by$ since $\bighomotopy$ is constant in the $s$-direction and therefore $\rank(d\bighomotopy) = \rank(d\homotopy) = 1 \leq 2$. \medskip

\textbf{Inverse Squares.}
Given a thin homotopy equivalence class of surfaces $[\bX] \in \thinsurface(\V)$ with a representative $\bX:[0,1]^2 \to \R^d$, the horizontal and vertical inverses are exactly $[\bX^{-h}], [\bX^{-v}] \in \thinsurface(\V)$. If $\bX \sim_{\thinhom} \bY$, then we have $\bX^{-h} \sim_{\thinhom} \bY^{-h}$ and $\bX^{-v} \sim_{\thinhom} \bY^{-v}$ by inverting the thin homotopy between $\bX$ and $\bY$ in the same way as the horizontal and vertical inverses. One can check that $[\bX^{-h}]$ and $[\bX^{-v}]$ act as inverses with respect to horizontal and vertical concatenation,
\[
    [\bX] \concat_h [\bX^{-h}] = [1^h_{\bdy_l \bX}], \quad \quad [\bX] \concat_v [\bX^{-v}] = [1^v_{\bdy_b \bX}],
\]
since thin homotopy quotients out 2-dimensional retracings (\Cref{ex:2d_thin_homotopy}).

\section{Proofs for Young Surface Development} \label{apx:young_proofs}
In this appendix, we provide the details for proving~\Cref{thm:main_sd_holder}.

\begin{remark}
    In this appendix, we use the following notation. For a path $\bx: [0,1] \to V$, we denote $\bx_{s,t} = \bx_t - \bx_s$. For a surface $\bX: [0,1]^2 \to V$, we use a semicolon to denote evaluation at a point $\bX_{s;t}$ and denote increments by $\bX_{s_1, s_2; t} = \bX_{s_2; t} - \bX_{s_1; t}$ and $\bX_{s; t_1, t_2} = \bX_{s; t_2} - \bX_{s; t_1}$. Furthermore, we let $\bX_{\bullet; t}: [0,1] \to \V$ denote the path $s \mapsto \bX_{s;t}$ and $\bX_{s;\bullet} : [0,1] \to \V$ denote the path $t \mapsto \bX_{s;t}$.
\end{remark}

\subsection{Preliminaries}
We begin with a lemma for generalized H\"older functions, whose proof will be instructive for later arguments. All matrices are equipped with the Frobenius norm.
\begin{lemma} \label{lem:matrix_product_holder}
    Let $n,m,p \in \N$. Let $\bX^1, \bX^2 \in C^\rho([0,1]^2, \Mat_{n,m})$ and $\bY^1, \bY^2 \in C^\rho([0,1]^2, \Mat_{m,p})$ such that $\|\bX^i\|_\rho, \|\bY^i\|_\rho < \ell$ for some $\ell > 0$. Then the pointwise matrix product $\bZ^i = \bX^i \cdot \bY^i$ satisfies 
    \begin{align}
        \|\bZ^1\|_\rho \leq \|\bX^1\|_\rho \|\bY^1\|_\rho \quad \andd \quad \|\bZ^1 - \bZ^2\|_\rho \ls_{\rho, \ell} \|\bX^1 - \bX^2\|_\rho + \|\bY^1 - \bY^2\|_\rho.
    \end{align}
\end{lemma}
\begin{proof}
    Let $R = [s_1, s_2] \times [t_1, t_2]$. By expanding $\square_R[\bZ]$ (see~\cite[Lemma 5]{harang_extension_2021}), we get
    \begin{align*}
        \square_{R}[\bZ] = \bX_{s;t} \cdot \square_{R}[\bY] + (\bX_{s;t}  - \bX_{s;t'})(\bY_{s;t'} - \bY_{s';t'}) + (\bX_{s;t} - \bX_{s';t}) (\bY_{s';t} - \bY_{s';t'}) + \square_{R}[\bX] \bY_{s';t'}.
    \end{align*}
    Then, we get
    \begin{align*}
        |\bZ|_{\rho, (1,2)} \leq \|\bX\|_\infty \cdot |\bY|_{\rho, (1,2)} + |\bX|_{\rho, (2)} \cdot |\bY|_{\rho, (1)} + |\bX|_{\rho, (1)} \cdot |\bY|_{\rho, (2)} + |\bX|_{\rho, (1,2)} \cdot \|\bY\|_\infty \leq \|\bX\|_\rho \cdot \|\bY\|_\rho.
    \end{align*}
    The bounds for $|\bZ|_{\rho, (i)}$ are the same as the classical setting. The second inequality can be obtained from the first by the triangle inequality.
\end{proof}

Next, we will require the standard bounds for Young integrals, where we use the version from~\cite[Theorem 16]{harang_extension_2021}.
\begin{proposition} \label{prop:young_1d}
    For $\by \in C^{\rho}([0,1], L(\R^d, \R^e))$ and $\bx \in C^\rho([0,1], \R^d)$, with $\rho > \frac12$, we have
    \begin{align}
        \left|\int_s^t \by_r d\bx_r - \by_s (\bx_t - \bx_s)\right| \leq C \llb \by\rrb_\rho \cdot \|\bx\|_\rho |t-s|^{\rho}.
    \end{align}
\end{proposition}

\begin{proposition} \label{prop:young_2d}
    For $\bY \in C^\rho([0,1]^2, L(\R^d, \R^e))$ and $\bX \in C^\rho([0,1]^2, \R^d)$, with $\rho > \frac12$ and a rectangle $R = [s_1, s_2] \times [t_1, t_2]$, we have
    \begin{align}
        \left|\int_R \bY_r d\bX_r - \bY_s \square_R[\bX]\right| \leq C \llb \bY\rrb_\rho \cdot \|\bX\|_\rho |s_1 - s_2|^\rho |t_1 - t_2|^\rho.
    \end{align}
\end{proposition}

We consider Young differential equations, which are equations of the form 
\begin{align} \label{eq:1d_cde}
    d\by_t = V(\by_t) d\bx_t,
\end{align}
where $\bx \in C^\rho([0,1], \V)$ is the driving signal and $V: \R^e \to \Lin(\R^d, \R^e)$ is a vector field. We consider \emph{affine} vector fields  $V(x) = V_0 + V_1(x)$, where $V_0 \in \R^e$ is a constant and $V_1 \in \Lin(\R^e, \Lin(\R^d, \R^e)) \label{pg:lip_norm}$ is a linear vector field. We equip affine vector fields with the Lipschitz norm,
\begin{align}
    \|V\|_{\Lip} = \llb V \rrb_{\Lip} + \|V(0)\| \quad \text{where} \quad \llb V \rrb_{\Lip} = \sup_{x,y \in \R^e} \frac{\|V(x) - V(y)\|}{|x-y|}.
\end{align}
We denote the solution of~\Cref{eq:1d_cde} with initial condition $\by_0$ by
\begin{align}
    P^V(\bx, \by_0) \in C^\rho([0,1], \V).
\end{align}
We require some basic results on such equations and use~\cite{galeati_nonlinear_2023} as the main reference. This studies nonlinear Young equations, but we will state results in the linear setting (with affine vector fields), which arises as special cases.

\begin{proposition}{\cite[Theorem 3.9]{galeati_nonlinear_2023}} \label{prop:gronwall}
    Let $\rho > \frac12$, $V : \R^e \to L(\R^d, \R^e)$ be a Lipschitz vector field, $\by \in C^\rho([0,1], \R^e)$, and $\bx \in C^\rho([0,1], \R^d)$. Let $v, \ell > 0$ such that $\|V\|_{\Lip}< v$ and $|\bx|_\rho < \ell$. If $\by$ satisfies
    \begin{align}
        |\by_t - \by_{s}| \leq \left|\int_s^t V(\by_r) d\bx_r\right| + D |t-s|^\rho,
    \end{align}
    for some $D> 0$, then,
    \begin{align}
        |\by|_\rho \ls_{\rho, v, \ell} \|V\|_{\Lip} |\bx|_\rho (\|\by_0\|+D) + D.
    \end{align}
\end{proposition}

\begin{proposition}{\cite[Theorem 3.9, Corollary 3.13]{galeati_nonlinear_2023}}\label{prop:young_cde_bound}
    Let $\bx \in C^\rho([0,1], \V)$ with $\rho > \frac12$, and $V : \R^e \to L(\R^d, \R^e)$ is a Lipschitz vector field. There exists a constant $C>0$ depending only on $\rho$ such that for all initial conditions $\by_0 \in \R^e$, there exists a unique solution $P^V(\bx, \by_0)$ to~\Cref{eq:1d_cde} and it satisfies
    \begin{align}
        \| P^V(\bx, \by_0)\|_\rho \ls_{\rho} \exp\left(C \|V\|_{\Lip}^2 \|\bx\|_{\rho}^2\right) \|\by_0\|.
    \end{align}
\end{proposition}

\begin{proposition}{\cite[Theorem 3.14]{galeati_nonlinear_2023}} \label{prop:young_cde_stability}
    Let $\bx, \bx' \in C^\rho([0,1], \V)$, where $\rho > \frac12$, $V, V'$ be Lipschitz vector fields. Suppose $\|\bx\|_\rho, \|\bx'\|_\rho < \ell$ and $\|V\|_{\Lip}, \|V'\|_{\Lip} < v$. Let $\by_0, \by'_0 \in \V$ be two initial conditions. Then, the solutions $\by, \by'$ satisfy
    \begin{align}
        \llb P^V(\bx, \by_0) - P^{V'}(\bx', \by'_0)\rrb_\rho \ls_{\rho,v,\ell} \|\by_0 - \by'_0\| + \|\bx - \bx'\|_\rho + \|V - V'\|_{\Lip}.
    \end{align}
\end{proposition}

\subsection{2D Regularity and Stability of Path Development}
Here, we prove some preliminary results about the 2D regularity and stability of the path development along paths in a surface. In particular, given $\bX \in C^\rho([0,1]^2, \V)$ and $V \in \Lin(\R^e, \Lin(\V, \R^e))$, we define $Q^V(\bX) : [0,1]^2 \to \R^e$ by
\begin{align} \label{eq:path_holonomy_surfaceP}
    Q^V_{s;t}(\bX) = P^V(\bx^{s,t}, q) = \int_0^t V(Q^V_{0;r}(\bX)) d\bX_{0;r} + \int_0^s V(Q^V_{r;t}(\bX)) d\bX_{r;t} \quad \text{with} \quad Q^V_{0;0}(\bX) = q,
\end{align}
where $\bx^{s,t}$ is the tail path~\Cref{eq:tail_path}.
Our first task is to show that $Q^V(\bX) \in C^\rho([0,1]^2, \V)$.

\begin{proposition} \label{prop:pdproc_holder}
    Let $\bX \in C^\rho([0,1]^2, \V)$ and $V \in \Lin(\R^e \Lin(\R^e, \R^e))$, then $Q^V(\bX) \in C^\rho([0,1]^2, \V)$.
\end{proposition}
\begin{proof}
    Let $R = [s_1, s_2] \times [t_1, t_2]$ and $\bz_t \coloneqq P^V_{t}(\bX_{0;\bullet}, q)$. Then, by~\Cref{prop:young_cde_stability}, we have
    \begin{align} \label{eq:pdproc_rect_bound}
        \left|\square_R[Q^V(\bX)]\right| &= \left|P^V_{s_1 s_2}(\bX_{\bullet; t_2}, \bz_{t_2}) - P^V_{s_1, s_2}(\bX_{\bullet; t_1}, \bz_{t_1})\right| \\
        & \ls (\|\bz_{t_2} - \bz_{t_1}\| + \|\bX_{\bullet; t_2} - \bX_{\bullet; t_1}\|_\rho) \cdot|s_2 - s_1|^\rho \\
        & \ls( \llb \bz\rrb_\rho |t_2 - t_1|^\rho + |\bX|_{\rho,(1,2)} |t_2 - t_1|^\rho) |s_2 - s_1|^\rho.
    \end{align}
    Then, since $\llb \bz\rrb_\rho$ can be bound by~\Cref{prop:young_cde_bound}, we have $|Q^V(\bX)|_{\rho, (1,2)} < \infty$. Next, for the condition in the $(1)$ direction, we have
    \begin{align}
        |Q^V_{s_2; t} - Q^V_{s_1;t}| = |P^V_{s_1, s_2}(\bX_{\bullet; t}, \bz_t)| \ls \llb P^V(\bX_{\bullet; t}, \bz_t) \rrb_\rho |s_2 - s_1|^\rho,
    \end{align}
    which can also be bound by~\Cref{prop:young_cde_bound}. Finally, in the $(2)$ direction, we have
    \begin{align}
        |Q^V_{s; t_2} - Q^V_{s; t_1}| &= |P^V_{s,0}(\bX_{\bullet; t_2}, \bz_{t_2}) - P^V_{s,0}(\bX_{\bullet; t_1}, \bz_{t_1}) - (\bz_{t_2} - \bz_{t_1})| \\
        &\ls (\|\bz_{t_2} - \bz_{t_1}\| + \|\bX_{\bullet; t_2} - \bX_{\bullet; t_1}\|_\rho) + \|\bz_{t_2} - \bz_{t_1}\|,
    \end{align}
    which can be bound in the same way as~\Cref{eq:pdproc_rect_bound}. Thus, $Q^V(\bX) \in C^\rho([0,1]^2, \V)$.
\end{proof}

Thus, we have defined a map
\begin{align} \label{eq:Q}
    Q: C^\rho([0,1]^2, \V) \times \Lin(\R^e, \Lin(\V, \R^e)) \to C^\rho([0,1]^2, \R^e),
\end{align}
and we will now study the continuity of this map with respect to both variables.

\begin{proposition} \label{prop:pdproc_stability_surface}
    Let map $Q$ is locally Lipschitz with respect to $C^\rho([0,1]^2, \V)$.
\end{proposition}
\begin{proof}
    Let $\bX, \bY \in C^\rho([0,1]^2, \V)$ and $R = [s_1, s_2] \times [t_1, t_2]$. Denote
    \begin{align*}
        \wbX_{s;t} = Q^V_{s;t}(\bX), \quad \wbY_{s;t} = Q^V_{s;t}(\bY), \quad \bq_s = \bX_{s;t_1} - \bY_{s;t_1} - \bX_{s;t_2} + \bY_{s;t_2}, \quad \wbq_s = \wbX_{s;t_1} - \wbY_{s;t_1} - \wbX_{s;t_2} + \wbY_{s;t_2}
    \end{align*}
    By direct computation, we get $\square_R[Q^V(\bX) - Q^V(\bY)] = \wbq_{s_1, s_2}$. Then, by expanding the definition of $Q$, using the linearity of $V$, and expanding as in~\cite[Lemma 5]{harang_extension_2021} (see proof of~\Cref{lem:matrix_product_holder}), we obtain the following bound for $|\wbq_{s_2} - \wbq_{s_1}|$, 
    \begin{align*}
        &\left|\int_{s_1}^{s_2} V(\wbX_{s;t_1}) d\bq_s + V(\wbq_s) d\bX_{s;t_2} + V(\wbX_{s;t_1} - \wbY_{s;t_1}) d(\bY_{s; t_1} - \bY_{s;t_2}) + V(\wbY_{s;t_1} - \wbY_{s;t_2}) d(\bX_{s;t_2} - \bY_{s;t_2})\right| \\
        & \quad \leq C |s_2 - s_1|^\rho  + \int_{s_1}^{s_2}V(\wbq_s) d\bX_{s;t_2}
    \end{align*}
    where
    \begin{align}
        C = v\left(\llb \tbX_{\bullet; t_1} \rrb_\rho \|\bq\|_\rho + \llb \wbX_{\bullet; t_1} - \wbY_{\bullet; t_1} \rrb_{\rho} \|\bY_{\bullet; t_1} - \bY_{\bullet; t_2}\|_\rho +  \llb \wbY_{\bullet; t_1} - \wbY_{\bullet; t_2}\rrb_\rho \|\bX_{\bullet; t_2} - \bY_{\bullet; t_2}\|_\rho  \right).
    \end{align}
    Then, by~\Cref{prop:gronwall}, we have
    \begin{align}
        \llb \wbq\rrb_\rho \ls_{\rho, v, \ell}  \|\wbq_0\|+C.
    \end{align}
    By definition, $\wbq_0 = P^V_{t_1, t_2}(\bX_{0; \bullet}, q) - P^V_{t_1, t_2}(\bY_{0; \bullet}, q)$, so by~\Cref{prop:young_cde_stability}, we have
    \begin{align}
        \|\wbq_0\| \leq \llb P^V_{t_1, t_2}(\bX_{0; \bullet}, q) - P^V_{t_1, t_2}(\bY_{0; \bullet}, q) \rrb_\rho |t_1 - t_2|^\rho \ls_{\rho, v, \ell} \|\bX - \bY\|_\rho |t_1 - t_2|^\rho.
    \end{align}
    Similarly, we have $\|\bq\|_\rho \ls\|\bX-\bY\|_{\rho} |t_1 - t_2|^\rho$ and
    \begin{align}
       \llb \wbX_{\bullet; t_1} - \wbY_{\bullet; t_1} \rrb_{\rho}, \|\bX_{\bullet; t_2} - \bY_{\bullet; t_2}\|_\rho  \ls \|\bX - \bY\|, \quad \|\bY_{\bullet; t_1} - \bY_{\bullet; t_2}\|_\rho, \llb \wbY_{\bullet; t_1} - \wbY_{\bullet; t_2}\rrb_\rho  \ls |t_1 - t_2|^\rho,
    \end{align}
    so that $C \ls_{\rho, v, \ell} \|\bX - \bY\|_\rho |t_1 - t_2|^\rho$. Thus, we have
    \begin{align}
        \left|\square_R[Q^V(\bX) - Q^V(\bY)]\right| \ls_{\rho, v, \ell} \|\bX- \bY\|_{\rho} |s_1 - s_2|^\rho |t_1 - t_2|^\rho.
    \end{align}
    The bounds for the $(1)$ and $(2)$ directions can be performed similarly.
\end{proof}

\begin{proposition} \label{prop:pdproc_stability_vf}
    Let map $Q$ is locally Lipschitz with respect to $\Lin(\R^e, \Lin(\R^d, \R^e))$.
\end{proposition}
\begin{proof}
    Let $\bX \in C^\rho([0,1]^2, \V)$, $ V, W \in \Lin(\R^e, \Lin(\R^d, \R^e))$ and $R = [s_1, s_2] \times [t_1, t_2]$. Denote
    \begin{align*}
        \wbX_{s;t} = Q^V_{s;t}(\bX), \quad \wbY_{s;t} = Q^W_{s;t}(\bX), \quad \wbq_s = \wbX_{s;t_1} - \wbY_{s;t_1} - \wbX_{s;t_2} + \wbY_{s;t_2}
    \end{align*}
    Note that $\square_R[Q^V(\bX) - Q^W(\bX)] = \wbq_{s_1,s_2}$. Similar to~\Cref{prop:pdproc_stability_vf}, we expand $|\wbq_{s,s'}|$ to get
    \begin{align}
         &\Big| \int_{s_1}^{s_2} V(\wbq_r)d \bX_{s;t_1} + (V - W)(\wbX_{s;t_1,t_2}) d\bX_{s;t_1,t_2} + (V-W)(\wbY_{s;t_1, t_2}) d\bX_{s;t_1} + W(\wbX_{s;t_2} - \wbY_{s;t_2}) d\bX_{s; t_1, t_2} \Big| \nonumber \\
        & \quad \leq C|s_2 - s_1|^\rho + \int_{s_1}^{s_2} V(\wbq_s)d \bX_{s;t_1},
    \end{align}
    where
    \begin{align}
        C = \|V - W\|_{\Lip}\left(\llb \wbX_{\bullet; t_2}\rrb_\rho \|\bX_{\bullet;t_1, t_2}\|_\rho +  \llb \wbY_{\bullet;t_1,t_2}\rrb_\rho \|\bX_{\bullet; t_1}\|_\rho \right) + v \llb \wbX_{\bullet; t_2} - \wbY_{\bullet; t_2}\rrb_\rho \|\bX_{\bullet; t_1, t_2} \|_\rho.
    \end{align}
    Then, by~\Cref{prop:gronwall}, we have
    \begin{align}
        \llb \wbq \rrb_\rho \ls_{\rho, v, \ell} \|\wbq_0\| + C.
    \end{align}
    First, by definition of $\wbq$ and using~\Cref{prop:young_cde_stability}, we obtain
    \begin{align}
        \|\wbq_0\| \leq \llb P^V_{t_1, t_2}(\bX_{0; \bullet}, q) - P^W_{t_1, t_2}(\bX_{0; \bullet}, q)\rrb_\rho |t_1 - t_2|^\rho \ls_{\rho, v, \ell} \|V - W\|_{\Lip} |t_1 - t_2|^\rho 
    \end{align}
    Next, we consider $C$, and note that by applying~\Cref{prop:young_cde_stability} again, we obtain
    \begin{align*}
        \llb \wbX_{s_1, s_2; t_2} - \wbY_{s_1, s_2; t_2}\rrb_\rho& = \llb P^V_{s_1, s_2}(\bX_{\bullet; t_2}, \wbX_{0;t_2}) - P^W_{s_1, s_2}(\bX_{\bullet; t_2}, \wbY_{0;t_2})\rrb_\rho \ls_{\rho, v, \ell} \|V - W\|_{\Lip} + \|\wbX_{0;t_2} - \wbY_{0;t_2}\|.
    \end{align*}
    By the same reasoning, we have $\|\wbX_{0;t_2} - \wbY_{0;t_2}\| \ls_{\rho, v, \ell} \|V - W\|_{\Lip}$, and furthermore that
    \begin{align}
        \llb \wbX_{\bullet; t_2}\rrb_\rho, \|\bX_{\bullet; t_1}\|_\rho \ls 1 \quad \andd \quad \|\bX_{\bullet; t_1, t_2}\|_\rho, \llb \wbY_{\bullet;t_1, t_2}\rrb_\rho \ls|t_1 - t_2|^\rho.
    \end{align}
    Putting these all of these bounds together, we obtain
    \begin{align}
        |Q^V(\bX) - Q^W(\bY)|_{\rho, (1,2)} \ls_{\rho, v, \ell} \|V - W\|_{\Lip} |s_1 - s_2|^\rho |t_1 - t_2|^\rho.
    \end{align}
    The bounds for the $(1)$ and $(2)$ directions can be performed similarly.
\end{proof}

Consider a modified definition of $Q$ where we flip the order of the horizontal and vertical components of the tail path. Given $\bX \in C^\rho([0,1]^2, \V)$ and $V \in \Lin(\R^e, \Lin(\V, \R^e))$, we define $\wQ^V(\bX): [0,1]^2 \to \R^e$ by
\begin{align} \label{eq:Q_flipped}
    \wQ^V_{s;t}(\bX) = \int_0^s V(\wQ^V_{r;0}(\bX)) d\bX_{r;0} + \int_0^t V(\wQ^V_{s;r}(\bX)) d\bX_{s;r} \quad \text{with} \quad \wQ^V_{0;0}(\bX) = q.
\end{align}
This is defined by applying $Q$ to the reflection of $\bX$ along $s=t$, thus implying the following.

\begin{corollary} \label{cor:pdproc_flipped}
    The map $\wQ: C^\rho([0,1]^2, \V) \times \Lin(\R^e, \Lin(\V, \R^e)) \to C^\rho([0,1]^2, \R^e)$ from~\Cref{eq:Q_flipped} is well defined and locally Lipschitz.
\end{corollary}

\subsection{Young Surface Development}

Next, we apply the results in the previous section to study the regularity and stability of the area process. For a surface $\bX \in C^\rho([0,1]^2, \V)$, we let $\bx^{s,s';t,t'}: [0,1] \to \V$ denote the boundary path of $\bX$ restricted to the rectangle $[s,s'] \times [t,t']$. Then, the area process is defined as the level 2 path signature
\begin{align} \label{eq:area_process_young}
    A_{s;t}(\bX) = \int_{\Delta^2} d\bx^{0,s;0,t}_{u_1} \otimes d\bx^{0,s;0,t}_{u_2} \in \Lambda^2 \V,
\end{align}
which is anti-symmetric since $\bx^{0,s;0,t}$ is a loop. 

\begin{proposition} \label{prop:areaproc_holder_cont}
    The area process $A: C^\rho([0,1]^2, \V) \to C^\rho([0,1]^2, \Lambda^2 \V)$ defined in~\Cref{eq:area_process_young} is well-defined and locally Lipschitz. 
\end{proposition}
\begin{proof}
    First, we can decompose the boundary loop as $\bx^{0,s;0,t} = \wbx^{s,t} \concat (\bx^{s,t})^{-1}$, where
    \begin{align} 
        \bx^{s,t}_u \coloneqq \left\{ \begin{array}{cl}
            \bX_{0;2ut} & : u \in [0,1/2] \\
            \bX_{(2u-1)s; t } & : u \in (1/2, 1]
        \end{array} \right.
        \quad \andd \quad
        \wbx^{s,t}_u \coloneqq \left\{ \begin{array}{cl}
            \bX_{2us;0} & : u \in [0,1/2] \\
            \bX_{s; (2u-1)t } & : u \in (1/2, 1]
        \end{array} \right.
    \end{align}
    Second, we note the path signature with respect to $\bx^{s,t}$ and $\wbx^{s,t}$ is a Young CDE in the form of $Q(\bX)$ and $\wQ(\bX)$ respectively. In particular, we can show that
    \begin{align}
        A_{s;t}(\bX) = \wQ_{s;t}(\bX) - Q_{s;t}(\bX). 
    \end{align}
    Then, we obtain the result by applying~\Cref{prop:pdproc_holder},~\Cref{prop:pdproc_stability_surface} and~\Cref{cor:pdproc_flipped}.
\end{proof}

We can put the above results together to define surface development in the Young regime.

\begin{proposition} \label{prop:Z_holder}
    The map $Z: C^\rho([0,1]^2, \V) \times \cM^{n,m,p} \to C^\rho([0,1]^2, \cmgl_1^{n,m,p})$ defined by
    \begin{align}
        Z^\bomega_{s;t}(\bX) \coloneqq \int_0^t \int_0^s F^\alpha(\bx^{s',t'}) \cdot \gamma \cdot F^\beta(\bx^{s',t'})^{-1} \, dA_{s';t'}(\bX),
    \end{align}
    where $\bomega = (\cona, \conb, \conc) \in \cM^{n,m,p}$, is well defined and locally Lipschitz in both variables. 
\end{proposition}
\begin{proof}
    First, we note that for $\bX \in C^\rho([0,1]^2, \V)$, the path development with respect to the tail paths correspond to the $Q$ map defined in~\Cref{eq:Q}; in particular, $F^\alpha(\bx^{s,t}) = Q^\alpha_{s;t}(\bX)$ with $\R^e = \Mat_{n+m, n+m}$ and $F^\beta(\bx^{s,t}) = Q^\beta_{s;t}(\bX)$ with $\R^e = \Mat_{n+p, n+p}$.
    
    Next, fix $v, \ell > 0$, and let $L>0$ such that $\det(F^\beta(\bx^{s,t})): [0,1]^2 \to \R$ is bounded below by $L$ for all $\|\bX\|_\rho \leq \ell$. Note that matrix inversion is Lipschitz on such matrices with Lipschitz constant $2/L^2$. Thus, from~\Cref{prop:pdproc_holder} and \Cref{lem:matrix_product_holder}, we have
    \begin{align}
        T^\bomega(\bX) = F^\alpha(\bx^{s,t}) \cdot \gamma \cdot F^\beta(\bx^{s',t'})^{-1} \in C^\rho([0,1]^2, \Lin(\Lambda^2\V, \cmgl_1^{n,m,p})).
    \end{align}
    Then, since $A(\bX) \in C^\rho([0,1]^2, \Lambda^2 \V)$ we can perform 2D Young integration (\Cref{prop:young_2d}), and thus $Z^\bomega(\bX) \in C^\rho([0,1]^2, \cmgl_1^{n,m,p})$. \medskip

    For continuity, we start by considering two surfaces $\bX, \bY \in C^\rho([0,1]^2, \V)$ with $\|\bX\|_\rho , \|\bY\|_\rho < \ell$. Then, by~\Cref{lem:matrix_product_holder}, we have
    \begin{align*}
        \|T^\bomega(\bX) - T^\bomega(\bY)\|_\rho \ls_{\rho, \ell} \|F^\alpha(\bx^{s,t}) - F^\alpha(\by^{s,t})\|_\rho \|\gamma\|_\rho + \|F^\beta(\bx^{s,t})^{-1} - F^\beta(\by^{s,t})^{-1}\|_\rho \ls_{\rho, v, \ell} \|\bX - \bY\|_\rho.
    \end{align*}
    Next, let $\bomega_1, \bomega_2 \in \cM^{n,m,p}$ and by applying~\Cref{lem:matrix_product_holder} again, we obtain
    \begin{align*}
        \|T^{\bomega_1}(\bX) - T^{\bomega_2}(\bX)\|_\rho &\ls_{\rho, v, \ell} \|F^{\alpha_1}(\bx^{s,t}) - F^{\alpha_2}(\bx^{s,t})\|_\rho + \|\gamma_1 - \gamma_2\| + \|G^{\beta_1}(\bx^{s,t})^{-1} - G^{\beta_2}(\bx^{s,t})\|_\rho\\
        & \ls_{\rho, v, \ell} \|\bomega_1 - \bomega_2\|.
    \end{align*}
    Putting this together, we get $\|T^{\bomega_1}(\bX) - T^{\bomega_2}(\bY)\|_\rho \ls_{\rho, v, \ell} \|\bX - \bY\|_\rho + \|\bomega - \bomega_2\|$. Finally, for $Z$,
    \begin{align*}
        \|Z^{\bomega_1}(\bX) - Z^{\bomega_2}(\bY)\|_\rho &\ls_{\rho, \ell}\left\|\int \big(T^{\bomega_1}(\bX) - T^{\bomega_2}(\bX)\big) dA(\bX)\right\|_\rho + \left\| \int T^{\bomega_2}(\bX) d\big(A(\bX) - A(\bY)\big)\right\|_\rho \\
       &\ls_{\rho, v, \ell} \|\bX - \bY\|_\rho + \|\bomega_1 - \bomega_2\|
    \end{align*}
    by using the standard Young bound~\Cref{prop:young_2d}, the above bounds, and~\Cref{prop:areaproc_holder_cont}.
\end{proof}

Now, we can use this to prove the main result.

\begin{proof}[Proof of~\Cref{thm:main_sd_holder}]
    From~\Cref{eq:msh_decomp1}, surface development for $\bX \in C^\rho([0,1]^2, \V)$ is defined as a 1D CDE with respect to a linear growth vector field, where the driving signal is $Z_{1;\bullet}(\bX)$, which is $\rho$-H\"older by~\Cref{prop:Z_holder}. The solution is well-defined by~\Cref{prop:young_cde_bound}. Furthermore, the local Lipschitz continuity is given by using the stability of Young differential equations in~\Cref{prop:young_cde_stability}, and the fact that the driving signal $Z^\bomega_{1; \bullet}(\bX)$ is locally Lipschitz with respect to both the surface $\bX$ and the $2$-connection $\bX$.
\end{proof}

\section{Direct Sums and Tensor Products of 2-Connections}

\subsection{Direct Sums and Inclusions} \label{apxsec:direct_sums}
First, we consider direct sums of two matrix 2-connections $\con_1 = (\cona_1, \conb_1, \conc_1) \in \cM^{n_1, m_1, p_1}$ and $\con_2 = (\cona_2, \conb_2, \conc_2) \in \cM^{n_2, m_2, p_2}$. The \emph{direct sum} of these two 2-connections is a 2-connection $\con_1 \op \con_2 \coloneqq \con = (\cona, \conb, \conc)$ valued in $\cmgl^{n_1,m_1,p_1} \oplus \cmgl^{n_2,m_2,p_2}$.
This is viewed as the automorphisms of the 2-vector space\footnote{Note that the $\phi$ defined here does \emph{not} follow the convention from the rest of the paper that it has the form $\phi = \psmat{I & 0 \\ 0 & 0}$.}
\[
    \cV^{n_1,m_1,p_1} \oplus \cV^{n_2, m_2, p_2} \coloneqq \R^{n_1 + m_1} \oplus \R^{n_2 + m_2} \xrightarrow{\phi} \R^{n_1 + p_1} \oplus \R^{n_2 + p_2}, \quad \phi = \pmat{\phi_1 & 0 \\ 0 & \phi_2},
\]
where $\phi_1$ and $\phi_2$ are the linear maps defining the 2-vector spaces $\cV^{n_1,m_1,p_1}$ and $\cV^{n_2, m_2, p_2}$ respectively. 
We define the 2-connection $\con =(\cona, \conb, \conc)$ by
\[
    \cona^i \coloneqq \pmat{\cona^i_1 & 0 \\ 0 & \cona^i_2}, \quad \conb^i \coloneqq \pmat{\conb^i_1 & 0 \\ 0 & \conb^i_2}, \quad \conc^{i,j} \coloneqq \pmat{\conc^{i,j}_1 & 0 \\ 0 & \conc^{i,j}_2}.
\]
Then, given a path $\bx: [0,1] \to \V$, the path development of the direct sum connection is given by
\[
    F^{\cona}(\bx) = \pmat{F^{\cona_1}(\bx) & 0 \\ 0 & F^{\cona_2}(\bx)} \quad \andd \quad F^{\conb}(\bx) = \pmat{F^{\conb_1}(\bx) & 0 \\ 0 & F^{\conb_2}(\bx)}.
\]

\begin{lemma} \label{lem:direct_sum_sh}
    Given a surface $\bX: [0,1]^2 \to \V$, the surface development of the direct sum connection $\con = (\cona,\conb,\conc)$ as defined above is
    \begin{equation} \label{eq:direct_sum_holonomy}
        H^{\con}(\bX) = \pmat{H^{\con_1}(\bX) & 0 \\ 0 & H^{\con_2}(\bX)}.
    \end{equation}
\end{lemma}
\begin{proof}
    This can be verified by direct computations, as all matrices are diagonal. 
\end{proof}

Next, we consider inclusions $\cmgl^{n_1, m_1,p_1} \hookrightarrow \cmgl^{n_2, m_2, p_2}$ for $(n_1, m_1, p_1) \leq (n_2, m_2, p_2)$, along with the corresponding crossed module of Lie groups $\cmGL^{n_1, m_1, p_1} \hookrightarrow \cmGL^{n_2, m_2, p_2}$. In particular, given
\begin{align} \label{eq:directsum_embed1}
    \pmat{A & 0 \\ B & C}, \pmat{A & D \\ 0 & E} \in \cmgl^{n_1,m_1,p_1}_{1}, \quad \pmat{R & S \\ T & U} \in \cmgl^{n_1,m_1,p_1}_2
\end{align}
we define
\begin{align} \label{eq:directsum_embed2}
    \pmat{A' & 0 \\ B' & C'}, \pmat{A' & D' \\ 0 & E'} \in \cmgl^{n_2,m_2,p_2}_{1}, \quad \pmat{R' & S' \\ T' & U'} \in \cmgl^{n_2,m_2,p_2}_2,
\end{align}
where $A' = \pmat{A & 0 \\ 0 & 0} \in \Mat_{n_2,n_2}$, and similarly for all other blocks. 

\begin{lemma}
    Let $n,m,p \in \N$, $\con \in \cM^{n,m,p}$, and $\ell \in \Mat_{n+p,n+m}$ be a 2-connection valued in $\cmgl^{n,m,p}$. Then, for $N \geq n,m,p$, there exists a matrix 2-connection $\con' \in \cM^N$ and $\ell' \in \Mat_{2N, 2N}$ such that
    \[
        \langle \ell, H^{\con}(\cdot) \rangle = \langle \ell', H^{\con'}(\cdot) \rangle.
    \]
\end{lemma}
\begin{proof}
    We define $\con' \in \cM^N$ and $\ell' \in \Mat_{2N,2N}$ by the embedding $\cmgl^{n, m,p} \hookrightarrow \cmgl^{N,N,N}$ as described in~\Cref{eq:directsum_embed1} and~\Cref{eq:directsum_embed2}. Then, the two surface holonomies will have the form
    \[
        H^\con(\bX) = \pmat{R & S \\ T & U} \quad \andd \quad H^{\con'}(\bX) = \pmat{ R' & S' \\ T' & U'},
    \]
    where each block has the form $R' = \psmat{R & 0 \\ 0 & 0}$. Because $\ell'$ has the same form, we obtain the result.
\end{proof}

\subsection{Tensor Product of Surface Developments} \label{apxsec:2_tensor_product}

We begin by recalling the tensor product for path development. Let $V$ and $W$ be finite dimensional vector spaces and suppose $\cona_V \in \Lin(\V, \fgl(V))$ and $\cona_W \in \Lin(\V, \fgl(W))$ are two 1-connections. The tensor product of these is a connection $\cona \in \Lin(\V, \fgl(V \otimes W))$ defined by
\[
    \cona \coloneqq \cona_V \otimes I + I \otimes \cona_W. 
\]
Let $F^{\cona_V} : \thinpath(\V) \to \GL(V)$ and $F^{\cona_W} : \thinpath(\V) \to \GL(W)$ be the respective path development functors, and the path development of $\cona$ is the tensor product 
\[
    F^\cona  = F^{\cona_V} \otimes F^{\cona_W} :\thinpath(\V) \to \GL(V \otimes W).
\]
Indeed, given some $\bx \in C^\infty([0,1],\V)$, we have
\begin{align*}
    \frac{d (F^{\cona_V}_s(\bx) \otimes F^{\cona_W}_s(\bx))}{ds} &= \frac{dF^{\cona_V}_s(\bx)}{ds} \otimes F^{\cona_W}_s(\bx) + F_s^{\cona_V}(\bx) \otimes \frac{dF^{\cona_W}_s(\bx)}{ds} = (F^{\cona_V}_s(\bx) \otimes F^{\cona_W}_s(\bx)) \cdot \alpha(\bx).
\end{align*}
Thus, the tensor product of path development functors is given by the tensor product of the corresponding linear automorphisms (1-morphisms) of the vector spaces $V$ and $W$. This leads to a convenient algebra structure on linear functionals of path development. Indeed, suppose $\ell_V$ and $\ell_W$ are linear functionals of $\GL(V)$ and $\GL(W)$ respectively. Then,
\[
    \langle \ell_V, F^{\cona_V}(\bx)\rangle \cdot \langle \ell_W, F^{\cona_W}(\bx)\rangle = \langle \ell_V \otimes \ell_W, F^{\cona_V}(\bx) \otimes F^{\cona_W}(\bx) \rangle = \langle \ell_V \otimes \ell_W, F^{\cona}(\bx) \rangle,
\]
provides an algebra structure on the collection of linear functionals of path development on matrix groups. This leads to variants of the characteristic function on path space which does not require the exponential~\cite{chevyrev_characteristic_2016,cuchiero_global_2023}, though it requires additional moment conditions. \medskip

Unfortunately, due to the structure of tensor products of 2-vector spaces, such a property does not extend to surface development. Let $\cV$ and $\cW$ be two 2-vector spaces, and $\con_{\cV}$ and $\con_{\cW}$ be two 2-connections valued in $\cmgl(\cV)$ and $\cmgl(\cW)$ respectively. Let $H^{\con_{\cV}} : \thinsurface(\V) \to \cmGL_1(\cV)$ and $H^{\con_{\cW}} :\thinsurface(\V) \to \cmGL_1(\cW)$ denote their respective surface holonomies. 
First, the tensor product of 2-vector spaces discussed in~\cite{baez_higher-dimensional_algebras_2004} does not correspond to the usual tensor product of chain complexes. It is shown in~\cite{loday_tensor_1998} that there exists a modified tensor product on 2-term chain complexes such that the equivalence between 2-vector spaces and 2-term chain complexes is symmetric monoidal.
In particular, the tensor product of two 2-vector spaces is given by the 2-vector space $\cV \wot \cW = (\cV \wot \cW)_1 \xrightarrow{\phi_{\wot}} (\cV \wot \cW)_0$ where
\begin{align*}
    (\cV \wot \cW)_1 &\coloneqq  V_1 \ot W_0 \oplus V_0 \ot W_1\quad \andd \quad (\cV \wot \cW)_0 \coloneqq V_0 \ot W_0. \label{eq:monoidal_2vect} 
\end{align*}
Here the linear map is given by $\phi_{\wot}  \coloneqq \pmat{\phi_V \ot I & I \ot \phi_W}$.
However, the main issue is that we do not have a component of the form
\[
    H^{\con_\cV}(\bX) \otimes H^{\con_\cW}(\bX) : \cV_0 \otimes \cW_0 \to \cV_1 \otimes \cW_1
\]
component. Indeed, the lack of this term is due to degree issues: since $\cV_1 \otimes \cW_1$ is of total degree 2, the map $H^{\con_\cV}(\bX) \otimes H^{\con_\cW}(\bX)$ is of degree 2, while chain homotopies are degree 1 by definition.

\section{Characterization of Surfaces using Polynomial Integrals} \label{apx:polynomial_integrals}

In this section, we prove the following result. Let $C^{\rho}_{\zax}([0,1]^2, \R^d)$ denote surfaces $\bX$ such that $\bX_{s,0} = \bX_{0,t} = 0$ for all $s,t \in [0,1]$.

\begin{proposition} \label{prop:2d_integral_characterization}
    For $\bX \in C^{\rho}_{\zax}([0,1]^2, \R^d)$ with $|\bX|_{\rho, (1,2)} > 0$, there exist $a,b \in \N$ and $i \in [d]$ such that 
    \[
        \int_{[0,1]^2} s^a t^b d\bX_{s,t}  \neq 0.
    \]
\end{proposition}

We begin with a preliminary lemma.

\begin{lemma} \label{lem:2d_poly_dense_in_1hol}
    Polynomials in $s^a t^b$ with $a, b \geq 1$ are dense in $C^{0,1}_{\zax}([0,1]^2, \R)$.
\end{lemma}
\begin{proof}
    Let $F \in C^{0,1}_{\zax}([0,1]^2, \R)$, so that $F$ can be expressed as $F_{s,t} = \int_0^t \int_0^s f_{s,t} ds dt$ for some continuous $f: [0,1]^2 \to \R$ (see~\cite[Proposition 1.40]{friz_multidimensional_2010} for the 1D statement). In particular, note that $|F|_{1, (1,2)} = |f|_\infty$. Then, since polynomials are dense in $C([0,1]^2, \R)$ with the uniform topology, there exists a sequence $p_n$ of polynomials such that $p_n \xrightarrow{\infty} f$. Then, define
    \[
        P_n(s,t) = \int_0^t \int_0^s p_n(s',t') ds' dt',
    \]
    which is a polynomial in $s^a t^b$ with $a, b \ge 1$, and we have $|P_n - F|_{1,(1,2)} \to 0$.
\end{proof}

\begin{proof}[Proof of~\Cref{prop:2d_integral_characterization}]
    Without loss of generality, we consider $\bX \in C^\rho([0,1]^2, \R)$. Furthermore, by~\Cref{lem:2d_poly_dense_in_1hol}, it suffices to show that there exists some $B \in C^{0,1}_{\zax}([0,1]^2, \V)$ such that $\int_{[0,1]^2} B_{s,t} \, d\bX_{s,t} \neq 0$.
    To build the $B$, we note that there exists a rectangle $R = [s,s'] \times [t,t']$ such that $\square_R[\bX] \neq 0$. Then, for any $\delta > 0$, let $B^\delta: [0,1]^2 \to \R$ be a smooth bump function which is equal to $1$ on $R$, and supported on $R^\delta = [s-\delta, s'+\delta] \times [t-\delta, t'+\delta]$. Then, note that
    \begin{align}
        \int_{[0,1]^2} B^\delta_{s,t} d\bX_{s,t} = \int_{R} d\bX_{s,t} + \int_{R^\delta - R} B^\delta_{s,t} d\bX_{s,t} = \square_R[\bX] + \int_{R^\delta - R} B^\delta_{s,t} d\bX_{s,t}.
    \end{align}
    Because $|B^\delta|_\infty \leq 1$, we note that $\left|\int_{R^\delta - R} B^\delta_{s,t} d\bX_{s,t}\right| \to 0$ as $\delta \to 0$. Thus, choosing $\delta$ sufficiently small, we obtain $\left|\int_{[0,1]^2} B^\delta_{s,t} d\bX_{s,t}\right| >0$.
\end{proof}

\bibliographystyle{plain}
\bibliography{matrix_holonomy}

\end{document}